\documentclass[a4paper, 11 pt, twoside]{amsart}
\usepackage{srollens-en}[2011/03/03]
\usepackage[left = 3.5cm, right = 3.5cm, headsep = 6mm,
footskip = 10mm, top = 35mm, bottom = 35mm, footnotesep=5mm, headheight =
2cm]{geometry}
\usepackage{graphicx}
\usepackage[expansion=false]{microtype}
\usepackage{booktabs, multirow,  tabularx}

\usepackage[usenames, dvipsnames]{xcolor}

\usepackage{tikz, tikz-cd}
\usepackage{enumitem}
\setlist[description]{leftmargin=0cm,  labelindent=\parindent}

\usepackage[unicode,bookmarks, pdftex, backref = page]{hyperref}
\hypersetup{colorlinks=true,citecolor=NavyBlue,linkcolor=NavyBlue,urlcolor=Orange, pdfpagemode=UseNone, breaklinks=true}

\newtheoremstyle{special-example}
  {}
  {}
  {}
  {\parindent}
  {\bfseries}
  {:}
  { }
  {}
  \theoremstyle{special-example}
\newtheorem{example}[equation]{Example}

\renewcommand{\tilde}{\widetilde}

\newcommand{\pp}{\mathbb P}

\newcommand{\OO}{\mathcal O}

\newcommand{\wt}{\widetilde}

\newcommand{\inv}{^{-1}}

\newcommand{\I}{\mathrm{i}}

\title{Gorenstein stable surfaces with $K_X^2 = 1$ and $p_g>0$}

\author{Marco Franciosi}
\address{Marco Franciosi\\Dipartimento di Matematica\\Universit\`a di Pisa \\Largo B. Pontecorvo 5\\I-56127  Pisa\\Italy}
\email{franciosi@dm.unipi.it}

\author{Rita Pardini}
\address{Rita Pardini\\Dipartimento di Matematica\\Universit\`a di Pisa \\Largo B. Pontecorvo 5\\I-56127  Pisa\\Italy}
\email{pardini@dm.unipi.it}

\author{S\"onke Rollenske}
\address{S\"onke Rollenske\\FB 12/Mathematik und Informatik\\
Philipps-Universit\"at Marburg\\
Hans-Meerwein-Str. 6\\
35032 Marburg\\
Germany}
\email{rollenske@mathematik.uni-marburg.de}

\begin{document}
\begin{abstract}
In this paper we 
consider Gorenstein stable surfaces with $K^2_X=1$ and positive geometric genus.
Extending classical results, we show that such surfaces admit a simple description as weighted complete intersection.

We exhibit  a wealth of surfaces of all possible Kodaira dimensions that  occur as normalisations of Gorenstein stable surfaces with $K_X^2=1$; for $p_g=2$ this leads to a rough stratification of the moduli space.

Explicit non-Gorenstein examples show that we need further techniques to understand all possible degenerations.
\end{abstract}

\maketitle

\setcounter{tocdepth}{1}
\tableofcontents

\section{Introduction}	
This is the third  in a series of papers studying Gorenstein stable surfaces with $K_X^2=1$. Such surfaces are parametrized by (an open part of) the moduli space of stable surfaces $\overline\gothM_{K^2, \chi}$, a natural compactification of Gieseker's moduli space of canonical models of surfaces of general type $\gothM_{K^2, \chi}$. 
Unlike the case of curves, the  moduli space of stable   surfaces  is not obtained just by adding a boundary divisor  but it can have extra   irreducible/connected components. Also there are numerical invariants which can be  realized by stable surfaces but not by minimal surfaces; most notably, $K_X^2$ may not be an integer  if $X$ is only $\IQ$-Gorenstein and  the holomorphic Euler characteristic of a stable surface can be negative.

The classification of minimal surfaces with  $K^2_X=1$ and positive geometric genus is a classical topic, studied for example by Enriques,  Kodaira, Horikawa, Catanese, and Todorov (see \cite{Enriques-book, Horikawa2, catanese79,  catanese80, todorov80}). 
In this paper we 
consider Gorenstein stable surfaces with $K^2_X=1$ and positive geometric genus, recovering the standard embeddings in weighted projective space and the known results on pluricanonical maps on the one hand  and finding a detailed description of
singular ones (either normal or non-normal) on the other hand. 

Our first result says that the classical descriptions  extend uniformly to Gorenstein stable surfaces. 
\begin{custom}[Theorem \ref{thm: canonical-ring}]
 Let $X$ be a Gorenstein stable surface with $K_X^2=1$. 
 \begin{enumerate}
  \item If $p_g(X)=2$ then $X$ is canonically embedded as a hypersurface of degree $10$ in the smooth locus of $\IP(1,1,2,5)$.
  \item If $p_g(X)=1$ then $X$ is canonically embedded as a complete intersection  of bidegree $(6,6)$ in the smooth locus of $\IP(1,2,2,3,3)$.
 \end{enumerate}
\end{custom}
Note however that this no longer  holds true if we drop the Gorenstein assumption (see Section \ref{sec: non-Gorenstein}).

As a consequence, such surfaces are smoothable, and therefore the moduli space
$\overline\gothM_{1,3}^{(Gor)}$ of Gorenstein stable surfaces with $K^2=1$ and $\chi=3$ is irreducible and rational of dimension $28$, whilst 
the moduli space
$\overline\gothM_{1,2}^{(Gor)}$ of Gorenstein stable surfaces with $K^2=1$ and $\chi=2$ is an irreducible and rational variety of dimension 18
(see Corollary \ref{cor:moduli}).

The above explicit description entails control over the structure of pluricanonical maps, especially the bicanonical map. In case $p_g(X)=2$ the bicanonical map realizes $X$ as a double cover of the quadric cone in $\IP^3$ branched over a quintic section. 
In  Section \ref{sec: pg=2} we make a detailed study of the possible branch divisors resulting in a (rough) stratification of the moduli space $\overline\gothM_{1,3}^{(Gor)}$. As a byproduct we show with explicit examples  that the resolution of a Gorenstein stable surface with $K_X^2=1$ and $p_g=2$ can have arbitrary Kodaira dimension; this had been announced in \cite{FPR15a}. 
We also  give  some  examples of  non-Gorenstein stable surfaces with $K_X^2=1$ and $p_g=2$ which are not canonically embedded in $\IP(1,1,2,5)$.

It is worth remarking  that these surfaces play an important role in the construction of threefolds near the Noether line, see e.g. \cite{mengchen04}.

The case of surfaces with $K_X^2=1$ and $p_g(X)=1$ was intensively studied for some time as it provided the counterexamle for the local Torelli theorem on surfaces \cite{kynev77} (see also \cite{usui00} and references therein). Since for the general such surface the bicanonical map $\phi_2\colon X\to \IP^2$ is not a Galois covering, we cannot carry out a similarly detailed analysis.  In Section \ref{sec: pg=1} we construct some examples, again of all possible Kodaira dimensions, which show the possible variations already in the special case where the bicanonical map is a bi-double cover.

The more challenging case of numerical Godeaux surfaces $(K_X^2=\chi(\ko_X)=1)$ will be treated in a subsequent paper.

\subsection*{Acknowledgements}  The first and the second  author  are  members of GNSAGA of INDAM. The third author is grateful for support of the DFG through the Emmy Noether program and partially through SFB 701. The collaboration  benefited immensely from a visit of the third author in Pisa supported by GNSAGA of INDAM. This project was partially supported by PRIN 2010 ``Geometria delle Variet\`a Algebriche'' of italian MIUR. 

We are indebted to Stephen Coughlan for Remark \ref{rem: index 5} and to Christian B\"ohning for explaining us how to prove the rationality of $\gothM_{1,3}$.  The third author would like to thank Paolo Lella for help with Macaulay 2 and Anne Fr\"uhbis-Kr\"uger for a discussion on adjacencies of elliptic singularities.

\subsection*{Notations and conventions. }
We work exclusively with schemes of finite type over the complex numbers.
\begin{itemize}
\item A surface is a reduced, projective scheme  of pure dimension two but not necessarily irreducible or connected.
\item For a scheme $X$ which is Gorenstein in codimension 1 and $S_2$ we use the competing notations $mK_X$ and $\omega_X^{[m]}$ for multiples of canonical divisor,  respectively reflexive powers of the canonical sheaf.
\item Given a variety $Y$ and a line bundle $L\in \Pic(Y)$,  one defines the ring of sections $R(Y, L)= \bigoplus_{m\geq 0}H^0(mL)$; for $L=K_Y$, we have the \emph{ canonical ring}  $R(K_Y):=R(Y,K_Y)$.
\end{itemize}

\section{Stable surfaces and moduli spaces}

 In this section we recall some necessary notions and establish the notation that we need throughout the text. 
 Our main reference is \cite[Sect.~5.1--5.3]{KollarSMMP}. 

\subsection{Stable surfaces and log-canonical pairs}\label{ssec: definitions}
Let $X$ be a demi-normal surface, that is,  $X$ satisfies $S_2$ and  at each point of codimension one $X$ is either regular or has an ordinary double point.
We denote by  $\pi\colon \bar X \to X$ the normalisation of $X$. The conductor ideal
$ \shom_{\ko_X}(\pi_*\ko_{\bar X}, \ko_X)$
is an ideal sheaf  both in $\ko_X$ and $\ko_{\bar X} $ and as such defines subschemes
$D\subset X \text{ and } \bar D\subset \bar X,$
both reduced and of pure codimension 1; we often refer to $D$ as the non-normal locus of $X$.

\begin{defin}\label{defin: slc}
The  demi-normal surface $X$ is said to have \emph{semi-log-canonical (slc)}  singularities if it satisfies the following conditions: 
\begin{enumerate}
 \item The canonical divisor $K_X$ is $\IQ$-Cartier.
\item The pair $(\bar X, \bar D)$ has log-canonical (lc) singularities. 
\end{enumerate}
It  is called a stable  surface 
 if in addition $K_X$ is ample. In that case we define the geometric genus of $X$ to be $ p_g(X) = h^0(X, \omega_X) = h^2(X, \ko_X)$ and the irregularity as $q(X) = h^1(X, \omega_X) = h^1(X, \ko_X)$. 
A Gorenstein stable surface is a stable surface such that $K_X$ is a Cartier divisor.
\end{defin}

  Since  a demi-normal surface $X$ has at most double points in codimension one,  the map $\pi\colon \bar D \to D$ on the conductor divisors is generically a double cover and thus  induces a rational involution on $\bar D$. Normalising the conductor loci we get an honest involution $\tau\colon \bar D^\nu\to \bar D^\nu$ such that $D^\nu = \bar D^\nu/\tau$. By \cite[Thm.~5.13]{KollarSMMP},   the  triple $(\bar X,\bar D, \tau)$ determines $X$. 
  
  The   log-canonical pairs $(\bar X,\bar D)$ that can arise normalising  a Gorenstein stable surface $X$  with $K^2_X=1$ have been classified in   \cite[Thm.~1.1]{FPR15a} and are the following: 
\begin{itemize}
\item[($P$)] $\bar X=\pp^2$, $\bar D$ is a quartic.
\item[($dP$)] $\bar X$ is a Gorenstein Del Pezzo surface with $K^2_{\bar X}=1$ and $\bar D\in |-2K_{\bar X}|$.
\item[($E_-$)] $\bar X$ is obtained from a $\pp^1$-bundle $p\colon Y\to E$ over an elliptic curve by contracting a section $C_{\infty}$ with $C_{\infty}^2=-1$ and $\bar D$ is the image in $\bar X$ of a bisection of $p$ disjoint from $C_{\infty}$.
\item[($E_+$)] $\bar X=S^2E$, where $E$ is an elliptic curve and $\bar D$ is a trisection of the Albanese map $\bar X\to E$ with $p_a(\bar D)=2$.
\end{itemize}

In addition,  we have:  

\begin{thm}[\cite{FPR15b}, Prop. 4.2, \cite{FPR15a}, Thm. 3.6] \label{thm: types} 
Let $X$ be a Gorenstein stable surface with $K^2_X=1$ 
Then $0\leq \chi(X) \leq 3$ and moreover:
  \begin{enumerate}
  \item if  $\chi(X) = 0$ then $p_g(X)=0$ and $q(X)=1$
  \item if $\chi(X)>0$, then $q(X)=0$ and  $p_g(X)=\chi(X)-1$;
\item  if $\chi(X)=3$, then  $(\bar X, \bar D)$ is not of type $(E_+)$;
\item if $\chi(X)=1$, then  $(\bar X, \bar D)$ is not of type $(E_-)$; 
\end{enumerate}
\end{thm}

In what follows we do not use the classification of the pairs $(\bar X,\bar D)$ to describe  geometry of $X$ for $\chi(X)=2,3$, but we  analyse instead  the  canonical ring\ and the pluricanonical maps. However, from this analysis   we will be able to recover  examples of  all the  possible types of normalisation except a surface with $p_g(X)=1$ and normalisation $(E_+)$ (see Remark \ref{rem: E+}),  and thus keep a promise made in \cite[Sect. 4]{FPR15a}.

\subsection{Moduli spaces}

We will discuss surfaces in the following hierarchy of open inclusions of moduli spaces of surfaces with fixed invariants $a=K_X^2$ and $b=\chi(\ko_X)$:
\begin{center}
\begin{tikzpicture}[commutative diagrams/every diagram]
\matrix[matrix of math nodes, name=m] {
\gothM_{a,b}\\
\overline\gothM_{a,b}^{(Gor)}\\ \overline\gothM_{a,b}\\
};
\path[commutative diagrams/.cd, every arrow, every label]
(m-1-1) edge[commutative diagrams/hook] (m-2-1)
(m-2-1) edge[commutative diagrams/hook] (m-3-1);
\path (m-1-1) ++ (1,0) node[right] {= \text{Gieseker moduli space of surfaces of general type}};
\path (m-2-1) ++ (1,0) node[right] {= \text{moduli space of Gorenstein stable surfaces}};
\path (m-3-1) ++ (1,0) node[right] {= \text{moduli space of stable surfaces}};
\end{tikzpicture}
\end{center}
The openness of second inclusion follows from \cite[Cor.~3.3.15]{Bruns-Herzog}.
For the time being there is no  self-contained reference for the existence of the moduli space of stable surfaces with fixed invariants as a projective scheme, and we will not use this explicitly. A major obstacle in the construction is that in the  definition of the moduli functor one needs additional conditions beyond flatness to guarantee that invariants are constant in a family. For Gorenstein surfaces these problems do not play a role; we refer to \cite{kollar12} and the forthcoming book \cite{KollarModuli} for details.

\section{Canonical ring and  pluricanonical maps}\label{sec:pluri}
Here we compute the canonical ring of stable Gorenstein surfaces with $K^2=1$ and $p_g>0$. The upshot is that from this point of view  Gorenstein stable surfaces behave exactly like smooth minimal ones. 
The study of minimal surfaces with $K^2=1$ and $p_g=1,2$ goes back to Enriques  \cite[Chapter VIII]{Enriques-book} and Kodaira (compare  \cite[(2.1)]{Horikawa2}).
Later many authors studied such surfaces developing a complete picture (see \cite{todorov80, catanese79, catanese80, Horikawa2}). 
Therefore,
we give a very synthetic treatment, just stressing the points where a different argument is needed for the stable case.  

Our point of view is that the canonical ring can be recovered from the restriction to a canonical curve and  enables us to describe the canonical maps and to deduce some basic properties of the moduli spaces. It is thus important that in our case the canonical curves are sufficiently nice.

\begin{lem}[\cite{FPR15b}, Lem. 4.1] \label{lem: canonical-curve}
Let $X$ be a Gorenstein stable surface such that $K_X^2=1$ and let $C\in |K_X|$ be  a canonical curve.\\
Then  $C$ is an integral Gorenstein curve  with  $p_a(C)=2$. 
\end{lem}

\subsection{Half-canonical rings of Gorenstein curves of genus 2}\label{ssec: genus 2}
Let $C$ be a reduced and irreducible Gorenstein curve of genus 2 and  let $L\in \Pic(C)$ be a square root of $K_C$; in our application   $C$ is a canonical curve of $X$ and $L=K_X|_C$. 

The analysis of rings of sections of line bundles on curves goes back at least to Petri, and the following result is well known on smooth curves (see for example \cite[Sect.~4]{reid90}). 

To formulate the result in  a form coherent with Section  \ref{section:canonical ring}  we denote by $\bar S_1$ the polynomial ring $\IC[x_1,y, z]$ where $x_1$ has degree $1$, $y$ has degree $2$ and $z$ has degree $5$, and by 
$\bar S_2$ the polynomial ring $\IC[y_1, y_2,z_1,z_2]$ where $y_i$ has degree $2$ and $z_i$ has degree $3$ ($i=1,2$).

\begin{prop}\label{prop: curve-ring}
 Let $C$ be an integral Gorenstein curve with $p_a(C)=2$ and let $L\in \Pic(C)$ such that   $L^{\tensor 2}=\omega_C$.
\begin{enumerate}
\item If $h^0(L)=1$, then $R(L)\isom \bar S_1\slash (f)$, 
where $f=z^2+ y^5 +x_1^2 g(x_1, y)$ is  weighted homogeneous of degree 10.

\item If $h^0(L)=0$, then $R(L)\isom \bar S_2\slash (f_1,f_2)$, where  $f_1 = z_1^2+ c_1(y_1, y_2)$ and $f_2 =  z_2^2+ c_2(y_1, y_2)$ are weighted homogeneous of degree 6 and $c_1$, $c_2$ have no common factor.
\end{enumerate} 
\end{prop}

\begin{proof}
The main tools to prove statements  \refenum{i} and \refenum{ii} are
\begin{itemize}
\item[(a)] the Riemann-Roch theorem and Serre duality, used  to compute  $h^0(mL)$, $m\ge 1$ and to determine the base points of $|mL|$;
\item[(b)] the base point free pencil trick (see  \cite[chap. III \S 3 ]{ACGH}), used  to to show surjectivity of  multiplication maps of the form $H^0(aL)\tensor H^0(bL)\to H^0((a+b)L)$. 
\end{itemize}
Since both (a) and (b) hold for an irreducible Gorenstein curve (see for example \cite{franciosi13})
the degree of generators and relations can be determined verbatim as in the case of a smooth curve.

For \refenum{i} we can thus choose variables such that the unique relation is $z^2+h(x_1, y)$;
  it remains to prove that  $h(x_1,y)$ 
  is not divisible by $x_1^2$.  So assume by contradiction that this is the case: then the point $A=(0:1:0)$ lies on $C$ and therefore it is   singular for $C$. On the other hand, $A$ is the support of the zero locus on $C$ of the section $x_1\in H^0(L)$, but this is impossible since $L$ is a line bundle of degree 1.
 
 For \refenum{ii} we still need to show that $c_1$ and $c_2$ have no common factor. Assume for contradiction that both $c_1$ and $c_2$ are divisible by, say, $y_1$. Then the point $A=(0:1:0:0)$ lies on the curve, is singular for $C$ and is a base point of the $1$-dimensional system  $|3L|$. It follows that $A$ is a double point of $C$, the fixed part of  $|3L|$ is equal to $|2A|$ and the moving part $|M|$ of $|3L|$ is a linear system of dimension 1 and degree 1, contradicting the assumption that $C$ has genus 2. 
 \end{proof}

\subsection{The canonical ring }\label{section:canonical ring}

We now lift the descriptions of section rings from the previous section to Gorenstein stable surfaces, recovering in the case of a minimal surface of general type the previously known descriptions (see \cite[Ch.~VII, \S7]{BHPV} and \cite[\S 1 Prop. 6]{catanese79}). 

We denote by $S_1$ the polynomial ring $\IC[x_0, x_1,y, z]$ where $x_0,x_1$ have degree $1$, $y$ has degree $2$ and $z$ has degree $5$, and by 
$S_2$ the polynomial ring $\IC[x_0, y_1, y_2,z_1,z_2]$ where $x_0$ has degree $1$, $y_i$ has degree $2$ and $z_i$ has degree $3$ ($i=1,2$).

\begin{thm}\label{thm: canonical-ring}
 Let $X$ be a Gorenstein stable surface with $K_X^2=1$ and $p_g(X)>0$. Then   there are the following possibilities:
\begin{enumerate}
\item $q(X)=0$, $\chi(X)=3$ and $R(K_X)\isom S_1/(f)$, where  
\[f = z^2+y^5+g(x_0,x_1, y)\] 
is  weighted homogeneous of degree $10$ and $g$ does not contain $y^5$.  Hence  $X$ is canonically embedded as a hypersurface of degree $10$ in (the smooth locus of) $\IP(1,1,2,5)$. 
\item $q(X)=0$, $\chi(X)=2$ and  $R(K_X)\isom S_2/(f,g)$, where 
\begin{align*}
 f_1&=z_1^2+z_2x_0a_1(x_0,y_1, y_2)+ b_1(x_0,y_1, y_2),\\
 f_2&=z_2^2+z_1x_0a_2(x_0,y_1, y_2)+ b_2(x_0,y_1, y_2)
\end{align*}
are weighted homogeneous of degree 6. Hence   $X$ is canonically embedded as a complete intersection of bidegree $(6,6)$ in (the smooth locus of) $\IP(1,2,2,3,3)$. 
\end{enumerate}
\end{thm}
\begin{proof}
The possibilities for the invariants $p_g(X)$ and $q(X)$ have already been given in Theorem \ref{thm: types}.
Let $C\in |K_X|$ and set  $L=K_X|_C$, so that by adjunction we have $L^{\tensor 2}=\OO_C(K_C)$, and let   $x_0\in R(K_X)$ be a  section defining $C$. 
By Lemma \ref{lem: canonical-curve} the pair $(C,L)$ satisfies the hypothesis of Proposition \ref{prop: curve-ring}.

Consider  the usual restriction sequence 
\[\begin{tikzcd} 0\rar &\OO_X(mK_X-C)\rar{\cdot x_0}& \OO_X(mK_X)\rar & L^{\tensor m}\rar & 0
\end{tikzcd}\]
 Since   $q(X)=0$   by Theorem \ref{thm: types} and  $H^1(mK_X)=0$ for $m\geq 0$ by Kodaira vanishing \cite[Cor.\ 19]{liu-rollenske14}, 
  we see that the map $R(K_X)/(x_0) \to  R(L)$ is a surjection, hence an isomorphism.  In particular, $h^0(L)=p_g(X)-1$, 
so that the case $p_g=2$ corresponds to \refenum{i} of Proposition \ref{prop: curve-ring} and $p_g=1$ corresponds to case \refenum{ii}.
The claim about generators and relations is now obtained by lifting  the relations  of $R(L)$  to $R(K_X)$ and completing the squares in the lifted equations. 
 
 Assume that $p_g(X)=2$: the singular points of $\pp(1,1,2,5)$ are the points of coordinates $(0:0:1:0)$ and $(0:0:0:1)$. Neither of these belongs to $X$ since $f$ contains the monomials $y^5$ and $z^2$.
 
 Assume that $p_g(X)=1$: the singular points of $\pp(1,2,2,3,3)$ are the union of the two lines $\pp(2,2)$ and $\pp(3,3)$, which do not meet $X$ in view of the format of the equations and of the fact that by Proposition \ref{prop: curve-ring}, \refenum{ii}, the polynomials $b_1(0,y_1,y_2) $ and $b_2(0,y_1,y_2)$ have no common factor. 
  \end{proof}
Theorem \ref{thm: canonical-ring} gives immediately:
\begin{cor}\label{cor: smoothable}
Let $X$ be a stable Gorenstein surface with $K^2_X=1$ and $p_g(X)>0$ (equivalently, with $\chi(X)>1$, compare Theorem \ref{thm: types}).  
Then $X$ is smoothable.
\end{cor}

\begin{cor}\label{cor:moduli}~
\begin{enumerate}
\item The moduli space
$\overline\gothM_{1,3}^{(Gor)}$ of Gorenstein stable surfaces with $K^2=1$ and $\chi=3$ is irreducible and rational of dimension $28$.
 \item The moduli space
$\overline\gothM_{1,2}^{(Gor)}$ of Gorenstein stable surfaces with $K^2=1$ and $\chi=2$ is  irreducible and rational  of dimension 18.
\end{enumerate}
\end{cor}
\begin{proof}
By Corollary \ref{cor: smoothable}, the statements follow by the corresponding statements for minimal surfaces of general type
(see \cite[\S 3]{Horikawa2} for $\gothM_{1,3}$, \cite[Thm. 2.3]{catanese80} for $\gothM_{1,2}$). We could not track down a reference for $\gothM_{1,3}$ being rational, so here is a quick argument, explained to us by Christian B\"ohning: we fix the $y$ and $z$ coordinate  so that the equation is of the form
\[ z^2+y^5+f_4y^3+f_6y^2+f_8y+f_{10}=0,\]
where $f_i$ is a polynomial in $x_0, x_1$ of degree $i$. The remaining automorphisms are given by  $G=\mathrm {Gl}(2, \IC)$ acting on $x_0, x_1$. 

To conclude that $\gothM_{1,3}$ is rational it is thus sufficient to prove that the quotient of $V=H^0(\ko_{\IP^1}(4)\oplus \ko_{\IP^1}(6)\oplus \ko_{\IP^1}(8)\oplus \ko_{\IP^1}(10))$ by $G$ is rational. This follows from the so-called no-name-lemma \cite[Lem.~4.4]{cgr06} by looking at the projection $V\to H^0(\ko_{\IP^1}(6))$ because
\begin{enumerate}
 \item the quotient $H^0(\ko_{\IP^1}(6))/G$ is rational by \cite{bogomolov-katsylo85},
 \item the action of $G$ on $H^0(\ko_{\IP^1}(6))$ is generically free, i.e., the stabiliser in the generic point is trivial. \cite[II, \S7, Ex.\ 1 or Thm. 7.11]{AlgebraicGeometryIV}.
\end{enumerate}
This concludes the proof.

\end{proof}

\subsection{Pluricanonical maps}

Here we spell out some properties of  the pluricanonical maps of Gorenstein stable surfaces with $K^2=1$ and $p_g>0$. Such properties  are implicit in the description of the canonical ring given in Theorem \ref{thm: canonical-ring}.  As in the previous section, the known results for surfaces of general type  extend to  our case
(cf. \cite{todorov80, catanese79, catanese80, Horikawa2, BHPV}).

We denote by $\phi_m$ the $m$-canonical map of  the stable  Gorenstein surface $X$, given by the $m$-canonical system $|mK_X|$.
Recall that by Riemann-Roch and Serre-duality \cite[Cor.~3.2]{liu-rollenske13} for $m\ge 2$  one has $h^0(mK_X)=\chi(X)+\frac{m(m-1)}{2}$.

\begin{prop}\label{prop: pluri-chi=3} 
 Let $X$ be a Gorenstein stable surface with $K_X^2=1$ and $p_g(X)=2$. Then: 
\begin{enumerate}
\item $|K_X|$ has a simple base point $P$, which is smooth for $X$
\item $\phi_2$ is a finite degree 2 morphism,  $\phi_2(X)\subset \pp^3$ is  the quadric cone, image of the embedding of  $\IP(1,1,2)$ defined by ${|\ko(2)|}$, and  the branch locus of $\phi_2$ is the union of the vertex $O\in \phi_2(X)$ and of a degree   $5$ hypersurface section of $\phi_2(X)$ not containing $O$. The base point $P$ of $|K_X|$ is the only point mapped to $O$.
\item $\phi_3(X)$ is the embedding of $\IF_2$ as a  normal ruled surface of degree 4 in $\pp^5$, $\phi_3$ has degree 2 and is not defined at $P$. 
\item $\phi_4$ is the composition of  $\phi_2$  with the degree 2  Veronese embedding of $\phi_2(X)$ in $\pp^8$
\item $\phi_m$ is an embedding for $m\geq 5$. 
\end{enumerate}
\end{prop}
\begin{proof} We use the notation of Theorem \ref{thm: canonical-ring} for the generators of the canonical ring $R(K_X)$ and for the relations between them.

\refenum{i} A basis of $H^0(K_X)$ is given by $x_0, x_1$, hence the base locus of $|K_X|$ is the intersection of $X$ with $x_0=x_1=0$ and it consists just of the point $P=(0:0:-1: 1)$.  Since $K^2_X=1$ and all the canonical curves are reduced and irreducible by Lemma \ref{lem: canonical-curve}, $P$  is a simple   base  point and it is  smooth for $X$. 

\refenum{ii} Let $\pi\colon \IP(1,1,2)\into \IP^3$ be the embedding given by $|\ko(2)|$. A basis of $H^0(2K_X)$ is given by  $x_0^2, x_0x_1, x_1^2, y$, hence the bicanonical  map $\phi_2$  is  the projection from $\pp(1,1,2,5)$ to $\pp(1,1,2)$ composed with $\pi$. 
 Since  the point $(0:0:0:1)$ is not on $X$ by Theorem \ref{thm: canonical-ring}, the map $\phi_2$ is a finite degree 2 morphism.  The  branch points of $\phi_2$ different from the vertex $O=(0:0:1)\in \pp(1,1,2)$ are defined by the equation 
$f = y^5+g(x_0,x_1, y)$
  which  corresponds to a quintic section not containing $O$ in the embedding $\pi$.
The only point of $X$ that maps to  $O$ is the base point $P=(0:0:-1: 1)$
of $|K_X|$, hence $\phi_2$ is branched also on the vertex $O$.

\refenum{iii} and \refenum{iv} can be proven in a similar way.

\refenum{v} To show that the $m$-canonical map is an embedding for $m\geq 5 $ we note first that    for $m\ge 5$ the system  $|\ko_{\IP(1,1,2,5)}(m)|$ embeds the  locus $\IP(1,1,2,5)\setminus \pp(2,5)$. The surface $X$ intersects  $\pp(2, 5)$ only at the point $P=(0:0:-1: 1)$, which is the base point of the canonical system.  It is easy to check that for $m\ge 5$ the image of $\pp(2,5)$ via the map given by  $|\ko_{\IP(1,1,2,5)}(m)|$ is disjoint from the image of its complement  $\IP(1,1,2,5)\setminus \pp(2,5)$, hence $\phi_m$ is injective for $m\ge 5$. 
In addition, for every  $m\ge 5$, there exist a monomial $s$ of degree $m-1$ and a monomial $t$ of degree $m$
that do not vanish at $P$: then the $m$-canonical map  is  given locally by  $(\frac{sx_0}{t}, \frac{sx_1}{t},\dots)$ and therefore has injective differential  at $P$.
Indeed, the canonical curves defined by $x_0$ and $x_1$ intersect transversally at $P$, since they are distinct and irreducible  (Lemma \ref{lem: canonical-curve}) and $K^2_X=1$. 
\end{proof}

\begin{prop}\label{prop: pluri-chi=2}
 Let $X$ be a Gorenstein stable surface with $K_X^2=1$ and $p_g(X) = 1$. Then:
\begin{enumerate}
 \item $\phi_2\colon X\to \pp^2$ is a finite degree 4 morphism.
\item $\phi_3$ and  $\phi_4$ are  birational morphisms but do not embed.
\item $\phi_m$ is an embedding for $m\geq 5$.
\end{enumerate}
\end{prop}
\begin{proof}
 We use the notation of Theorem \ref{thm: canonical-ring} for the generators of the canonical ring $R(K_X)$. In particular, we have an isomorphism $H^0(\ko_{\IP(1,2,2,3,3)}(d))\isom H^0(d K_X)$ for $d\leq 5$. 

  \refenum{i} The bicanonical map  is induced by the projection $\IP(1,2,2,3,3) \dashrightarrow \IP(1,2,2)=\pp^2$. Looking at the equations of $X$ given in Theorem \ref{thm: canonical-ring}, one sees that $\phi_2$ is finite of degree four. 

\refenum{ii} The base-point-freeness of $|dK_X|$ follows by restriction from $|\ko_{\IP(1,2,2,3,3)}(d)|$, for $d\ge 3$ since $X$ does not meet the lines $\pp(2,2)$ and $\pp(3,3)$ by Theorem \ref{thm: canonical-ring}. 
The restriction of $|3K_X|$ to the canonical curve $C=\{x_0=0\}$ is spanned by $z_1$ and $z_2$, hence $\phi_3|_C$ is not birational and $\phi_3$ is not an embedding. 

To see that $\phi_3$ is birational, consider $X_0:=X\cap \{x_0\ne 0\}$ with affine coordinates $t_1=\frac{y_1}{x_0^2}, t_2=\frac{y_2}{x_0^2},t_3=\frac{z_1}{x_0^3},t_4=\frac{z_2}{x_0^3}$. The functions $t_1, \dots,  t_4$ are all pull-backs from the $3$-canonical image $\phi_3(X)$, hence $\phi_3$ is birational. 
One can argue exactly in the same way for $\phi_4$. 

\refenum{iii} Let $m\geq 5$.  We have already observed in \refenum{ii} that the $m$-canonical system  is base-point free.  Since the linear  system $|\ko_{\IP(1,2,2,3,3)}(m)|$ embeds the  complement of $\{x_0=0\}$ we only have to check that every subscheme $Z$ of length two with support intersecting the canonical curve $C$ is embedded. If $Z\subset 2C$ this follows because the restriction of the $m$-canonical system is very ample on $2C$ by the numerical criterion {\cite[Thm.~1.1]{CFHR}}.
That points on $C$ can be separated from points in the complement follows as above. 
\end{proof}

\section{Stratification of  $\overline\gothM_{1,3}^{(Gor)}$ and  beyond}\label{sec: pg=2}

In this section we examine more closely  the moduli spaces $\overline\gothM_{1,3}^{(Gor)}\subset \overline\gothM_{1,3}$.   The Gorenstein surfaces in the {boundary} $\overline\gothM_{1,3}\setminus \gothM_{1,3}$ are studied in detail. In addition, we show by explicit examples that a stable surface with $K_X^2=1$ and $\chi(X) = 3$ need neither be Gorenstein nor be canonically embedded in $\IP(1,1,2,5)$.

\subsection{The double cover construction}
In this section $X$ denotes a stable Gorenstein surface with $K^2_X=1$ and $\chi(X)=3$, or, equivalently, $p_g(X)=2$ and $q(X)=0$. 

We denote by $\kq\subset \pp^3$ the quadric cone, i.e.,  the  image of the embedding  of $\pp(1,1,2)$ given by  the system $|\OO_{\pp(1,1,2)}(2)|$, and by $O\in \kq$ the vertex.  

By Proposition \ref{prop: pluri-chi=3}, the bicanonical map is a double cover $\phi_2\colon X \to \kq$ branched on a divisor $\Delta\in |\OO_{\kq}(5)|$ such that $O\notin \Delta$. The divisor $\frac 12 \Delta$ is the so-called Hurwitz divisor of the cover (cf. \cite[Def.\ 2.4]{alexeev-pardini12}) and  $(\kq,\frac 12 \Delta)$ is  a log-canonical pair by \cite[Def.\ 2.5]{alexeev-pardini12}. 

Here we show that it is possible to reverse this construction:
\begin{prop}\label{prop: double-cover}
 Let $\Delta \in |\OO_{\kq}(5)|$ be  such that $(\kq,\frac12\Delta)$ is a log-canonical pair. Then:
\begin{enumerate}
\item There exists a unique    double cover $p\colon X\to \kq$ with Hurwitz divisor   $\frac 12 \Delta$
\item $X$ is a stable surface with $K^2_X=1$
\item the Cartier index of $X$ is equal to $1$ if $O\notin \Delta$ and it is equal to 2 otherwise
\item $X$ is normal if and only if $\Delta$ is reduced. 
\end{enumerate}
\end{prop}
\begin{proof}
\refenum{i} We write $\kq_0=\kq\setminus\{O\}$ and we denote by $\Delta_0$ the restriction of $\Delta$ to $\kq_0$.  Since $\Pic(\kq_0)=\IZ f$, where $f$ is the class obtained  by restricting a ruling of $\kq$, the divisor  $\Delta_0$ is linearly equivalent to $10f$, hence there exists a flat double cover $X_0\to \kq_0$ branched on $\Delta_0$.  The surface $X_0$ is demi-normal, since,  by definition of log pair,  all components of $\Delta$ have multiplicity $\le 2$. By taking the $S_2$-closure (cf. \cite[Lem.\ 1.2]{alexeev-pardini12}) one obtains a demi-normal double cover $p\colon X\to \kq$ with Hurwitz divisor $\frac 12 \Delta$. 

The cover $p$ is unique,  since it  is determined by its restriction to $\kq_0$ and $\Pic(\kq_0)$ has no torsion. 

\refenum{ii}, \refenum{iii}  One has $2K_X=p^*(2K_{\kq}+\Delta)=p^*(\OO_{\kq}(1))$, hence the Cartier index of $X$ is at most 2, $K_X$ is ample and $K_X^2=1$. 
Notice that $K_{X_0}$ is Cartier by construction, hence to decide whether $K_X$ is Cartier it is enough to examine the point $O$. Notice also that  $O$ is always in the branch locus of $p$, since a ruling of $\kq$ intersects  $\Delta_0$ in $5$ points and $5$ is an odd number. 
If $O\notin \Delta$, then $X\to \kq$ is normal over $\Delta$, hence it is smooth there and therefore $X$ is Gorenstein. The converse follows by Proposition \ref{prop: pluri-chi=3}, \refenum{ii}.

\refenum{iv} The surface $X$, being $S_2$ by construction, is normal if and only if it is smooth in codimension 1, if and only if $X_0$ is smooth in codimension 1, if and only if $\Delta$ is reduced. 
\end{proof}

\begin{rem}
As it will be apparent in the sequel, the utility of  Proposition \ref{prop: double-cover} lies mainly in the fact    that it reduces the analysis of the singularities of $X$ to the study of the curve $\Delta\subset \kq$. However the construction of $X$ as a double cover of $\kq$ is not easy to perform  in families, since it involves taking the $S_2$-closure of a cover of $\kq_0$. This difficulty can be avoided by using Theorem \ref{thm: canonical-ring}, \refenum{i}: in the notation there, the divisor $\Delta\subset \pp(1,1,2)$ is given by the term
$y^5+g(x_0,x_1, y)$
in the equation $f$ of $X\subset \pp(1,1,2,5)$, hence it is immediate   to construct a flat family having as (repeated) fibres  all the surfaces $X$ constructed as in Proposition \ref{prop: double-cover}.
\end{rem}

\subsection{Stratification of $\overline\gothM_{1,3}^{(Gor)}\setminus \gothM_{1,3}$}\label{ssec:strata}
We now describe more precisely $\overline\gothM_{1,3}^{(Gor)}\setminus \gothM_{1,3}$, namely we consider  ramified covers of $\kq$ as above where the Hurwitz divisor $\Delta$ does not contain the vertex.

If  $P\in \Delta$  is an isolated singularity, then one of the following occurs:
\begin{itemize}
\item $P$ is a {\em negligible} singularity:  $P$ is a double point or a triple point such that every point infinitely near to $P$ is at most double for the strict transform of $\Delta$. 
The preimage of $P$ in $X$ is a canonical singularity. 
\item $P$ is a quadruple point such that every point infinitely near to $P$ is at most double for the strict transform of $\Delta$. The preimage of $P$ in $X$ is an elliptic Gorenstein singularity of degree $2$.
\item $P$ is a $[3,3]$-point, namely $P$ is a triple point with an infinitely near triple point $P_1$ and all the points infinitely near to $P_1$ are at most double. The preimage of $P$ in $X$ is an elliptic Gorenstein singularity of degree $1$.
\end{itemize}
This can be checked by blowing up and analysing the cases  similar to what is done in \cite{liu-rollenske12a} or from the point of view of log-canonical threshold \cite[6.5]{Corti-Kollar-Smith}.

 We consider the (open) strata
\[ \gothN_{d_1, \dots, d_k} = \left\{X\in \overline\gothM_{1,3}^{(Gor)}
\left|\,\text{\begin{minipage}{.45\textwidth}
   $X$ is normal and has exactly $k$ elliptic singularites of degree $d_1\leq \dots\leq  d_k$ 
            \end{minipage}}\right.
 \right\}.
\]

\begin{table}[htb!]\caption{Irreducible  strata of normal surfaces in $\overline\gothM_{1,3}^{(Gor)}$}\label{tab: normal strata}
\begin{center}
 \begin{tabular}{cclc}
 \toprule
stratum & dimension & minimal resolution $\tilde X$ & $\kappa(\tilde X)$\\
\midrule
$\gothN_\varnothing=\gothM_{1,3}$& $28$ & general type & $2$\\  
$\gothN_2$ & $20$  & blow up of a K3-surface & $0$\\
$\gothN_1$ & $19$  & minimal elliptic surface with $\chi(\tilde X)=2$ & $1$\\
$\gothN_{2,2}$ & $12$  & rational surface& $-\infty$\\
$\gothN_{1,2}$ & $11$ & rational surface& $-\infty$\\
$\gothN_{1,1}^R$ & $10$ &rational surface& $-\infty$\\
$\gothN_{1,1}^E$ & $10$ &blow up of an Enriques surface & $0$\\
$\gothN_{1,1,2}$ & $2$ & ruled surface with $\chi(\tilde X)=0$& $-\infty$\\
$\gothN_{1,1,1}$ & $1$  & ruled surface with $\chi(\tilde X)=0$& $-\infty$\\
\bottomrule
\end{tabular}

\end{center}
\end{table}

\begin{prop}\label{prop: normal strata}
 The subset of $\overline\gothM_{1,3}^{(Gor)}$ parametrizing normal surfaces is stratified by the  irreducible and unirational strata $\gothN_{d_1, \dots, d_k}$
 given in Table \ref{tab: normal strata}.	
\end{prop}
\begin{proof}
We have explained above that it is sufficient to control the singularities of the branch divisor $\Delta\subset \kq$. For convenience we  work in $\IP(1,1,2)$ with coordinates $x_0, x_1, y$. Recall that the branch divisor is given by a polynomial in $H^0( \ko_{\IP(1,1,2)}(10))$, which is of  dimension $36$.  Note that any automorphism of $\IP(1,1,2)$ is of the form $(x_0, x_1, y)\mapsto (ax_0+bx_1, cx_0+dx_1, ey+q(x_0, x_1))$ where $ad-bc\neq 0$,  $e\in \IC^*$ and $q$ is a quadratic polynomial; fixing the degree 2 coordinate $y$ up to a multiple corresponds in the embedding in $\IP^3$ to fixing a hyperplane section not containing the vertex.

For each potential stratum  we first use the automorphim group to fix the the position of non-negligible singular points as far as possible and then impose singularities at these points, that is, look at $H^0(\ki(10))$ for an appropriate ideal sheaf $\ki$. For example, for $\gothN_2$ the ideal sheaf is $\ki= (x_1, y)^4$. Then three steps are needed to conclude:
\begin{enumerate}
 \item Compute the dimension of the automorphism group fixing the given configuration.
 \item Check how many conditions are imposed by the singularites, i.e., compute $h^0(\ki(10))$.
 \item Show that there is a reduced curve $\Delta$ in the linear system induced by $H^0(\ki(10))$ such that the pair $(\IP(1,1,2),\frac 1 2 \Delta)$ is slc and has exactly the prescribed singularities. 
 \end{enumerate}
The first is elementary. The second could be achieved either by explict computation in a computer algebra system or, in most cases, by showing $H^1(\ki(10))=0$ via Kawamata-Viehweg vanishing on a blow up. The third however requires the construction of examples, which for several strata required the computation of a basis for $H^0(\ki(10))$. Therefore we used the computer systematically also to determine $h^0(\ki(10))$ and do  not provide abstract arguments even in the cases where this is easily possible.  All computations have been carried out  using the computer algebra system Macaulay 2 \cite {M2}; a file containing the commented code is included in the arxiv source code.

The list of examples is given after the proof, so in the following we only address the first two questions. 
In each case the structure of the minimal resolution  follows from \cite[Thm.\ 4.1]{FPR15a} combined with \cite[Lem.\ 4.3]{FPR15a}) and $\chi(X)=3$, except for a surface in $\gothN_{1,1}$, which a priori could either be rational or an Enriques surface. We will see below that both cases occur.

Note that by \cite[Thm.\ 4.1]{FPR15a} there are no surfaces with more that $3$ elliptic singularities. Therefore we have to investigate branch curves with at most three  non-negligible singular points.

First of all let us consider $\gothN_1$, i.e., 
 assume that $\Delta$ has a point $[3,3]$ at $P=(1:0:0)$. The ruling through $P$ cannot be tangent to $\Delta$ in $P$ because otherwise it has to be contained in $\Delta$ which is excluded by Proposition \ref{prop: double-cover}. Thus we may assume that $\Delta$ is tangent to $H=\{y=0\}$ at $P$. A point $[3,3]$ with given tangent imposes $12$ independent conditions. Thus $\gothN_1$ is dominated by an open subset of a $23$-dimensional linear subspace of $\IP^{35}$ and thus it is unirational and irreducible of dimension $19$, because the subgroup of automorphisms fixing a point and a tangent direction is of dimension $4$.

Next let us  look at $\gothN_2$, i.e., assume that $\Delta$ has an ordinary quadruple point at $P$. This imposes $10$ independent conditions and thus $\gothN_2$ is dominated by an open subset of a $\IP^{25}$. The subgroup of automorphisms fixing a point has dimension $5$ and thus we get a unirational irreducible component of dimension $20$.

For the strata with two elliptic singularities we choose coordinates such that the corresponding singularities of $\Delta$ are at $P=(1:0:0)$ and $Q=(0:1:0)$. In the presence  of one point of type $[3,3]$ at $P$ we arrange in addition that $\Delta$ is tangent to $H=\{y=0\}$ at $P$. If $Q$ is of type $[3,3]$ as well then there are two cases: 
\begin{itemize}
 \item[(R)] $\Delta$ is not tangent to $H$ at $Q$. Then we can use a further automorphism to fix the tangent to $\Delta$ at $Q$, which has to be different from $\{y=0\}$. 
 \item[(E)] $\Delta$ is tangent to $H$ at $Q$. Then intersecting $\Delta$ and $H$ shows that $H$ is actually contained in $\Delta$. In this case we say that the two $[3,3]$ points have a matching tangent hyperplane.
\end{itemize}
We decompose $\gothN_{1,1}=\gothN_{1,1}^R\cup\gothN^E_{1,1}$ according to these two cases. It is straightforward to compute a bicanonical divisor of the minimal resolution of $X\in \gothN_{1,1}$ and to conclude that we get as a minimal model an  Enriques surface if $X$ is in $\gothN_{1,1}^E$ and a rational surface if $X$ is in $\gothN_{1,1}^R$.

The explicit computations show that two elliptic singularities impose independent conditions, unless we are in case $E$ where we have one condition less than expected.

The dimension of the group of automorphisms fixing $P$, $Q$ and a tangent direction at each point is equal to 1, thus $\gothN^R_{1,1}$ is unirational of dimension $10=35-2\cdot 12-1$. 

The dimension of the group of automorphisms fixing $P$, $Q$ and the hyperplane section $H$ is 2, thus $\gothN^E_{1,1}$ is unirational of dimension $10=35-2\cdot 12+1-2$, because the two $[3,3]$-points with matching tangent hyperplane impose only 23 conditions.

For the study of the stratum $\gothN_{1,2}$ we can adopt the same argument. To compute the dimension, since we do not need to fix the tangent direction at $Q$, we get  $\dim(\gothN_{1,2})= 35-12-10-2 = 11$. 

For the stratum  $\gothN_{2,2}$ with two quadruple points   we do not need to fix the tangent directions at all and compute $\dim \gothN_{2,2} = 35-2\cdot 10 - 3 = 12$.

To treat the remaining cases with three singular points, assume that we have a surface with three elliptic singularities. We choose coordinates such that the corresponding singularities of $\Delta$ are at $P$, $Q$, and $R=(1:1:0)$. If $\Delta$ has two quadruple points at $P$ and $Q$ then an additional quadruple point or  $[3,3]$ point at $R$ forces $\Delta$ to contain twice the hyperplane section $H$, because $H\cdot\Delta$ and $H\cdot(\Delta-H)$ would otherwise be too big; this is impossible. 

So we may assume that $\Delta$ is reduced and has $[3,3]$-points at $P$ and $Q$ and a point of multiplicity at least 3 at $R$.  If $H= \{y=0\}$ is tangent to $\Delta$ in $P$ then it is contained in $\Delta$, since otherwise  one would have $H\cdot\Delta>10$. Computing again  $H\cdot (\Delta-H)$ we  conclude that $2H\leq \Delta$, a contradiction. Thus the tangent in a  $[3,3]$-point is  neither $H$ nor a ruling of the cone. Note that the group of automorphisms of $\IP(1,1,2)$ fixing $P$, $Q$,  $R$,  and a tangent direction distinct from $H$  at $P$ is trivial. 

For $\gothN_{1,1,2}$ parametrize the tangent at $Q$ with a parameter $t$ and let $\ki_t$ be the ideal imposing a quadruple point at $R$, a $[3,3]$-point at $P$ with a fixed tangent and a $[3,3]$-point at $Q$, such that the infinitely near triple point is the chosen tangent direction, and let $V_t$ be the linear system induced by $H^0(\ki(10))$.
Then for $t$ general $V_t$ is a pencil, which however does not contain a reduced curve. Only when the two $[3,3]$-points have a matching tangent hyperplane the conditions imposed  are no longer independent and  we get a 2-dimensional linear system whose general member has the correct type of singularities.

Now consider $\gothN_{1,1,1}$. As long as no two of the $[3,3]$-points have a matching tangent hyperplane the conditions imposed by the singularities are independent and thus cannot be satisfied. In fact, it turns out that the only case in which there is a reduced curve satisfying the conditions is when two pairs of points have matching tangent hyperplanes, which gives us a pencil of solutions. 

\end{proof}
We conclude the discussion of normal Gorenstein degenerations  by giving an example in each  stratum. 
\begin{description}
 \item[Example in $\gothN_2$] An example with exactly one elliptic singularity of degree $2$ and some negligible singularities is given by $\Delta = H_1+H_2+H_3+H_4+H_5$ where $H_i$ are general hyperplane sections such that $H_1, \dots, H_4$ have a common point of intersection. 
 
 One can see the corresponding K3 surface in the following way. Consider a plane sextic $C$ which is the sum of four general lines and a conic. Fix a general point $P$ on the conic and let $L$ be the tangent in $P$. Let $S$ be the singular  K3 surface obtained as a double cover branched over $C$. The strict transform of $L$ is an elliptic curve with a node at the preimage of $P$. Blowing up the node and contracting the elliptic curve gives $X$. Alternatively, blow up twice at $P$ to separate $L$ and $C$ and then contract the strict transform of $L$, a $(-1)$-curve, to a smooth point $R$ and the exceptional $(-2)$ curve to an $A_1$ point. The result is the quadric cone $\kq$, and the strict transform of $C$ has a quadruple point at $R$. 
\item[Example in $\gothN_1$] An example with exactly one elliptic singularity of degree $1$ and some negligible singularities is given by $\Delta = H_1+H_2+H_3+C$ where the  $H_i$ are general hyperplane sections which are tangent at a point $P$ and $C$ is a general quadric section.

 Indeed, the pencil of hyperplane sections that are tangent to $\Delta$ at $P$ gives rise to a pencil of elliptic curves on $\tilde X$. Using the canonical bundle formula one can check that the elliptic fibration $\tilde X\to \IP^1$ has a unique multiple fibre of multiplicity two, which corresponds to twice the ruling through $P$ and the exceptional elliptic curve $E$ is a two-section ot the fibration.

 \item[Example in $\gothN_{2,2}$] An example with exactly two elliptic singularities of degree $2$ and some negligible singularities is given by $\Delta = H_1+H_2+H_3+H_4+H_5$ where the $H_1, \dots, H_4$ are general hyperplane sections through two points $P$ and $Q$ and $H_5$ is a hyperplane section not containing these points.
 
 The pencil of hyperplane section through $P$ and $Q$ pulls back to a pencil of rational curves on $X$.
 \item[Example in $\gothN_{1,2}$] An example with exactly two elliptic singularities, one of degree $1$ and one of degree $2$ is given by $z^2+y^5+x_1^4(x_0^6+y^3)+2y^4x_0^2=0$ in $\IP(1,1,2,5)$. Indeed,  the local equations   of the branch divisor at $P$, resp.\ $Q$, are $x_1^4(1+y^3)+2y^4(1+y)$, an ordinary quadruple point, resp.\ $x_0^6+y^3(1+y^2+yx_0^2)$, which is a point of type $[3,3]$. One can check that these are the only singular points. 
\item[Example in $\gothN^E_{1,1}$]  An example with exactly two elliptic singularities of degree $1$ and some negligible singularities is given by $\Delta = H_1+H_2+H_3+H_4+H_5$ where the $H_i$ are general hyperplane sections such that $H_1, H_2, H_3$ are tangent at a point $P$ and $H_3, H_4, H_5$ are tangent at a point $Q\neq P$. One can check that on the minimal resolution ${\tilde X}$ the bicanonical divisor  $2K_{\tilde X}$ is linearly equivalent to the strict transform of $p^* H_3$  which is twice a $(-1)$-curve. The minimal model thus has trivial bicanonical bundle and is an Enriques surface.

\item[Example in $\gothN^R_{1,1}$] Since the tangent directions of the two $[3,3]$ points do not match there is a certain genericity to the branch divisor, so we only could find a fairly complicated equation for $\Delta$:
\begin{multline*}
y^{5}+\left(5{x}_{0}^{2}+2 {x}_{0} {x}_{1}\right) y^{4}
+\left(19 {x}_{0}^{3} {x}_{1}+{x}_{0}^{2} {x}_{1}^{2} -{x}_{0} {x}_{1}^{3} \right)y^{3}
+\left(4 {x}_{0}^{4} {x}_{1}^{2}-3 {x}_{0}^{2} {x}_{1}^{4} \right)y^{2}\\\
-3 {x}_{0}^{3} {x}_{1}^{5} y
-{x}_{0}^{4} {x}_{1}^{6}
=0
\end{multline*}
 \begin{figure}\caption{Branch locus and surface in $\gothN_{1,1}^R$}\label{fig: N11R}
\includegraphics[width=4cm]{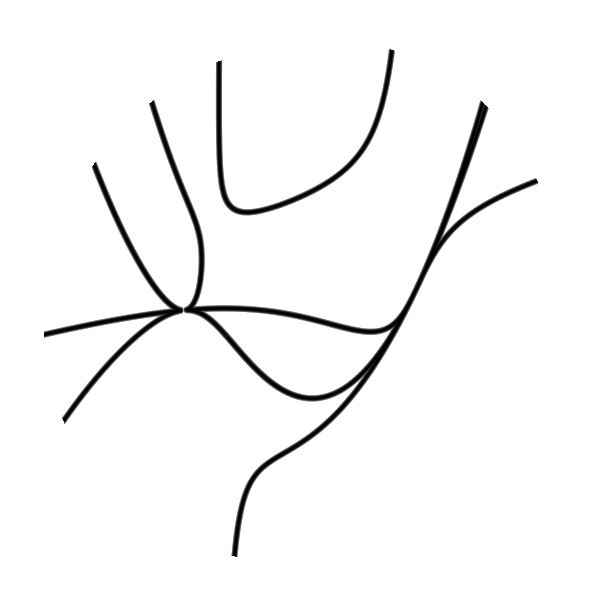}
\includegraphics[width=4cm]{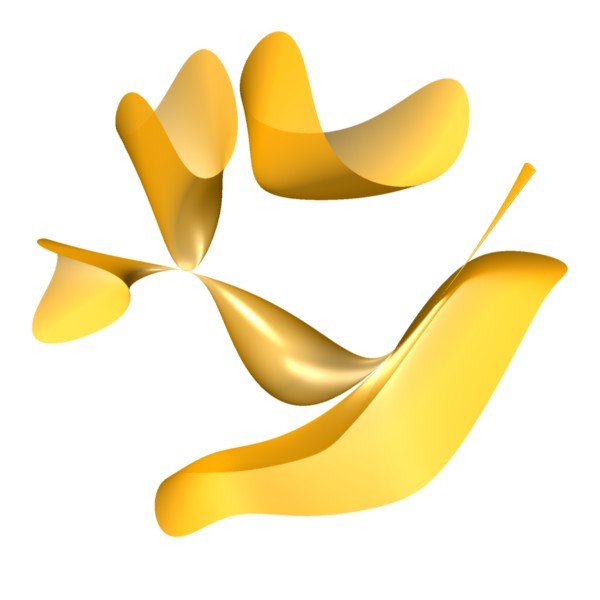}
\end{figure}

 \item[Example in $\gothN_{1,1,2}$] 
The branch divisor contains the hyperplane section tangent to the two $[3,3]$ points $P$ and $Q$. For a particular choice of this hyperplane the equation is a product 
 \[({x}_{0} {x}_{1}+y)(\alpha y^4+\beta f+\gamma g)=0\]
 where
 {\begin{gather*}    
   f={x}_{0}^{3} {x}_{1} y^{2}-2 {x}_{0}^{2} {x}_{1}^{2} y^{2}+{x}_{0} {x}_{1}^{3} y^{2}+{x}_{0}^{2} y^{3}-2 {x}_{0} {x}_{1} y^{3}+{x}_{1}^{2} y^{3}\\
   g=  {x}_{0}^{6} {x}_{1}^{2}-4 {x}_{0}^{5} {x}_{1}^{3}+6 {x}_{0}^{4} {x}_{1}^{4}-4 {x}_{0}^{3} {x}_{1}^{5}+{x}_{0}^{2} {x}_{1}^{6}+2 {x}_{0}^{5} {x}_{1} y-8 {x}_{0}^{4} {x}_{1}^{2} 
   y  +  12{x}_{0}^{3} {x}_{1}^{3} y\\
   -8{x}_{0}^{2} {x}_{1}^{4} y+2 {x}_{0} {x}_{1}^{5} y+{x}_{0}^{4} y^{2}-2 {x}_{0}^{2} {x}_{1}^{2} y^{2}+{x}_{1}^{4} y^{2}+4 {x}_{0}^{2} y^{3}-8 {x}_{0} {x}_{1} y^{3}+4 {x}_{1}^{2} y^{3}
 \end{gather*}
}
For a general choice of parameters, e.g. $\alpha=\beta=\gamma = 1$,  the curve has exactly the required singularities.

\begin{figure}\caption{Branch locus and surface in $\gothN_{1,1,2}$}\label{fig: N112}
 \includegraphics[width=4cm]{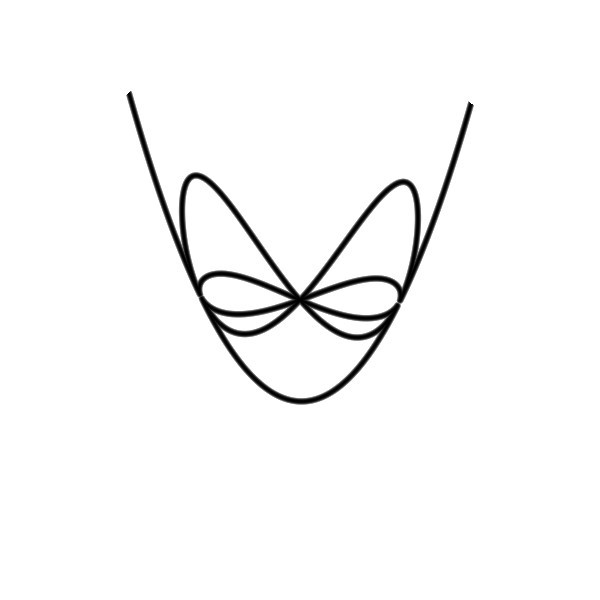}
\includegraphics[width=4cm]{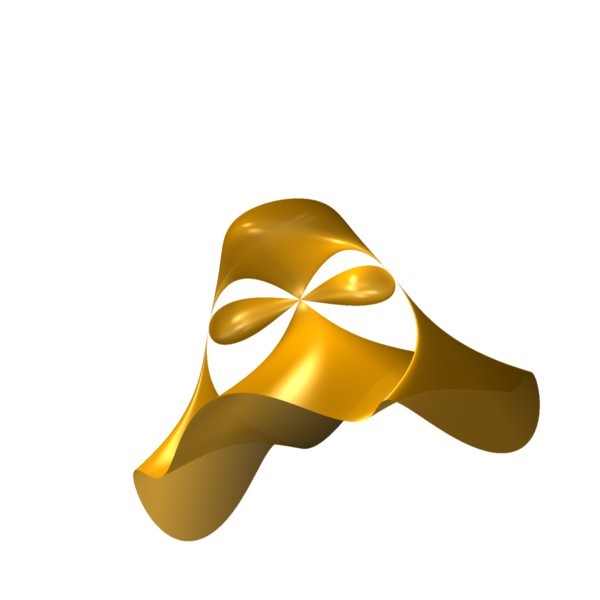}
\end{figure} 
 
 \item[Example in $\gothN_{1,1,1}$]
The branch divisor contains two hyperplane sections which each pass through two of the three $[3,3]$-points. Fixing those two with equations ${x}_{0} {x}_{1}+y$ and ${x}_{0} {x}_{1}-{x}_{1}^{2}+y$, there is a pencil of cubic sections of the cone generated by 
\[f = \left({{x}_{0}^{2} {x}_{1}-{x}_{0} {x}_{1}^{2}+2 {x}_{0} y-{x}_{1} y}\right)^{2}\]
and
\begin{multline*}
g={x}_{0}^{5} {x}_{1}+{x}_{0}^{4}
      {x}_{1}^{2}-5 {x}_{0}^{3} {x}_{1}^{3}+3 {x}_{0}^{2} {x}_{1}^{4}+{x}_{0}^{4} y+12 {x}_{0}^{3} {x}_{1} y
      -19 {x}_{0}^{2} {x}_{1}^{2} y+6 {x}_{0} {x}_{1}^{3} y\\
      +14 {x}_{0}^{2} y^{2}
      -13 {x}_{0} {x}_{1} y^{2}+3
      {x}_{1}^{2} y^{2}+y^{3}, 
\end{multline*}

whose general member has exactly the required singularities and tangencies.  
\begin{figure}\caption{Branch locus and surface in $\gothN_{1,1,1}$}\label{fig: N111}
\includegraphics[width=4cm]{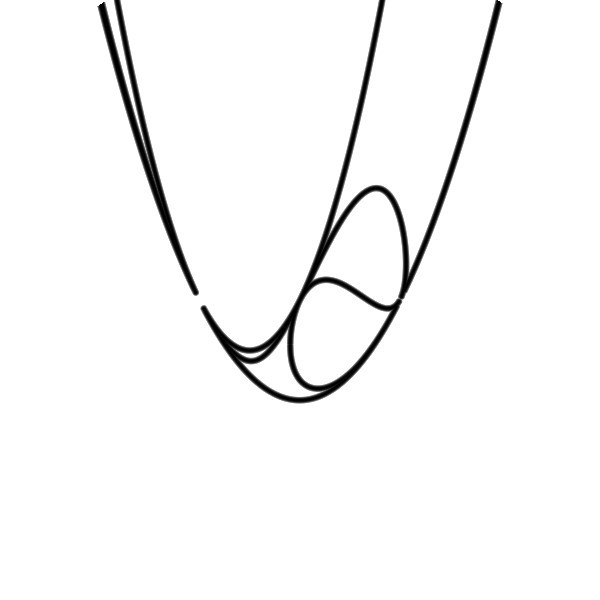}
\includegraphics[width=4cm]{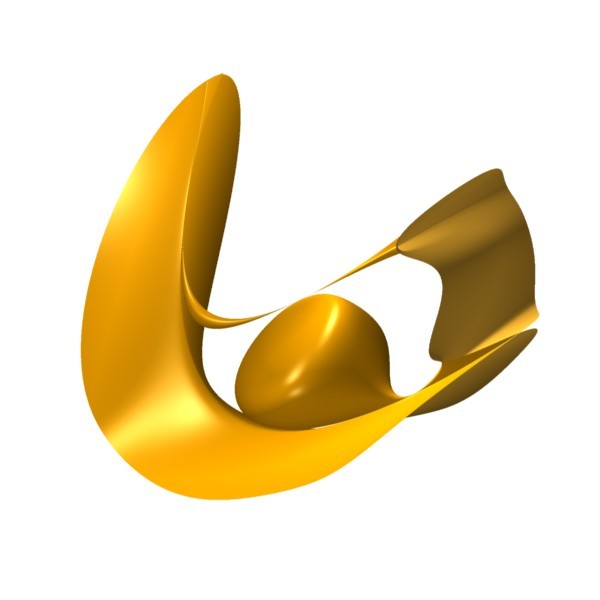}
\end{figure}

\end{description}

\begin{table}[htb!]\caption{Irreducible strata of non-normal surfaces in $\overline\gothM_{1,3}^{(Gor)}$}\label{tab: non-normal strata}
\begin{center}
 \begin{tabular}{ccl}
 \toprule
stratum & dimension  & minimal resolution $\tilde X$ \\
\midrule
$\gothd\gothP$ & $11$ &   del Pezzo surface of degree $1$\\
$\gothP$ & $4$ &    $\IP^2$\\
$\gothE$ & $2$ &  minimal ruled  surface with $\chi(\tilde X)=0$\\
\bottomrule
\end{tabular}
\end{center}
\end{table}

Next we describe the strata containing  non-normal Gorenstein surfaces. We refer to \S \ref{ssec: definitions} for the notation and the terminology. 
\begin{prop} \label{prop: non-normal-strata} The subset $\gothN\gothN \subset \overline \gothM_{1,3}^{(Gor)}$ of non-normal surfaces is equal to $\gothP\sqcup \gothd\gothP\sqcup\gothE$, where:
\begin{enumerate}
\item $\gothP\subset \overline \gothM_{1,3}^{(Gor)}$  consists of the non-normal surfaces with normalisation of type $(P)$,  and it   is irreducible and unirational   of dimension 4. The corresponding branch locus is a double quadric section plus a hyperplane section.
\item $\gothd\gothP\subset \overline \gothM_{1,3}^{(Gor)}$  consists of the non-normal surfaces with normalisation of type $(dP)$,  and it   is irreducible and unirational   of dimension 11. The corresponding  branch locus is a double hyperplane section plus a sufficiently general cubic section.
\item $\gothE\subset \overline \gothM_{1,3}^{(Gor)}$  consists of the non-normal surfaces with normalisation of type $(E_-)$,  and it   is irreducible and rational of dimension $2$. The corresponding branch locus is as in case $(dP)$ where the cubic section acquires a $[3,3]$ point.
\end{enumerate}
\end{prop}
\begin{proof}
 By Proposition \ref{prop: double-cover} a surface $X\in \overline \gothM_{1,3}^{(Gor)}$ is non-normal if and only if $\Delta$ is not reduced. Since $(\kq,\frac 12\Delta)$ is log-canonical, every component of $\Delta$ appears with multiplicity at most $2$, hence we may write $\Delta=\Delta_0+2\Delta_1$, with  $\Delta_1\in |\OO_{\kq}(k)|$ and  $\Delta_0\in |\OO_{\kq}(5-2k)|$, where $k=1$ or $k=2$. The normalisation $\bar X$ of $X$ is also a double cover of $\kq$, branched  over $ \Delta_0$ and the vertex $0\in \kq$.

Assume first $k=2$. Then $\Delta_0$ is a hyperplane section and $\bar X=\pp^2$, namely the normalisation of $X$ is of type $(P)$.

Assume now that $k=1$. If $\Delta_0$ has at most negligible singularities, then $\bar X$ is a del Pezzo surface of degree 1, hence $X$ is of type $(dP)$. 
If $\Delta_0$ has non-negligible singularities, then the only possibility is that $\bar X$ is of type $(E_-)$: in this case $\bar X$ has an elliptic singularity of degree $1$ and therefore $\Delta$ has a $[3,3]$-point.

For $k=2$, we may assume that  the section $\Delta_0$ is fixed. Letting $\Delta_1$ vary in an appropriate open subset $U\subset  |\OO_{\kq}(2)|$ one obtains all the surfaces of $\gothP$, hence $\gothP$ is unirational and irreducible. Since the subgroup of  automorphisms of $\kq$ that preserve $\Delta_0$ has dimension 4, it follows that  $\gothP$ has dimension $8-4=4$.

In case $k=1$ we may assume that $\Delta_1$ is fixed. Letting $\Delta_0$ vary in an appropriate open subset $U\subset  |\OO_{\kq}(3)|$ one obtains all the surfaces of $\gothd\gothP$, hence $\gothd\gothP$ is unirational and irreducible. Since the subgroup of  automorphisms of $\kq$ that preserve $\Delta_1$ has dimension 4, it follows that  $\gothd\gothP$ has dimension $15-4=11$.

To describe $\gothE$, notice that in this case the $[3,3]$-point $P\in \Delta_0$ does not lie on $\Delta_1$, because  $(\kq,\frac12 \Delta_0+\Delta_1)$ is a log-canonical pair. The subgroup   of the automorphisms of $\kq$ that preserve a fixed plane section $\Delta_1$ acts transitively on $\kq\setminus (\Delta_1\cup\{O\})$, so we may assume that the point $P$ is also fixed. In turn, the subgroup of the automorphisms of $\kq$ that preserve $\Delta_1\cup\{P\}$ fixes the infinitely near point $P_0$ corresponding to the ruling of $\kq$ containing $P$ and acts transitively on the set of points infinitely near to $P$ and distinct from $P_0$.  So it is enough to consider the divisors $\Delta_0\in|\OO_{\kq}(3)|$ with  triple points at $P$ and at a fixed infinitely near point $P_1$. Arguing as in the proof of Proposition \ref{prop: normal strata}, one sees that such divisors give an open subset of a $3$-dimensional linear subsystem of $|\OO_{\kq}(3)|$, so $\gothE$ is unirational and irreducible. (Actually, it is not hard to see that in this case $\Delta_0$ is the union of three plane sections passing through $P$ and $P_1$). 
The subgroup of automorphisms of $\kq$ that fix $\Delta_1$, $P$ and $P_1$ is $1$-dimensional, so $\gothE$ has dimension 2 and is rational. 
\end{proof}

In Figure \ref{fig: strata} we give a coarse picture of the relations of the strata, where we connect two strata if the lower one is contained in the closure of the upper one. It does not give much more than the obvious relations, because our stratification is note fine enough. Indeed, a general curve singularity of type $[3,3]$ cannot degenerate to an ordinary quadruple point, for example, because its Milnor number is too large; any degeneration to a quadruple point will not be ordinary. Thus to understand the intersections of the closures of the strata given in Table \ref{tab: normal strata} we would need a finer stratification, distinguishing the type of quadruple points occurring. For a description of the possible adjacencies see  \cite{brieskorn79}.

 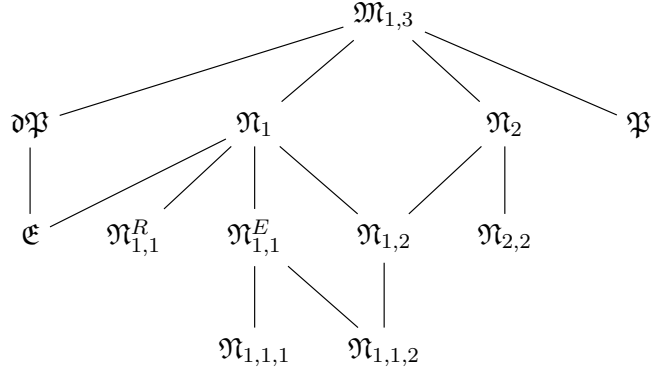
\begin{figure}\caption{Relation between some strata in $\overline\gothM_{1,3}^{(Gor)}$}\label{fig: strata}
\[
\begin{tikzcd}[arrows=dash, column sep = small]
 {}&&& \gothM_{1,3} \arrow{dlll} \arrow{dl}\arrow{dr}\arrow{drrr}\\
 \gothd\gothP\dar&& \gothN_1 \arrow{dll} \arrow{dl}\dar\arrow{dr} & & \gothN_2\dar\arrow{dl} & &\gothP\\
 \gothE&\gothN_{1,1}^R & \gothN^E_{1,1}\dar\arrow{dr} & \gothN_{1,2}\dar & \gothN_{2,2}\\
 &&\gothN_{1,1,1} & \gothN_{1,1,2}
 \end{tikzcd}
\]
\end{figure}

\begin{rem}
 The most degenerate surface $X_0$ that  we have encountered is the double cover of the quadric cone branched over $\Delta = H_1+2H_2+2H_3$ for three general hyperplane sections $H_1$, $H_2$, $H_3$, which is a particular surface of type $(P)$. The surface $X_0$ lies in the closure of $\gothE$, $\gothN_1$, $\gothN_{2,2}$, $\gothP$, and thus all strata lying above these in Figure \ref{fig: strata}. So we suspect $X_0$ is in the closure of every stratum that we considered. Finding a degeneration from the remaining strata with more than one elliptic singularity is however not obvious.
\end{rem}

\subsection{Some non-Gorenstein degenerations}\label{sec: non-Gorenstein}
We now discuss some increasingly general non-Gorenstein degenerations in $\overline\gothM_{1,3}$. More precisely, let $\gothD\gothC$ be the set of surfaces arising as double covers of the quadric cone as in Proposition \ref{prop: double-cover}. This family maps onto  an  irreducible  subset of $\overline\gothM_{1,3}$,  so that we have inclusions:

\[\overline\gothM_{1,3}^{(Gor)}\subset \gothD\gothC\subset \overline\gothM_{1,3}.\]
The examples given below will show that all the above inclusions are strict. 

 \begin{rem}\label{rem: index 5} One is tempted to consider also the subset $\gothH\gothS$ of slc hypersurfaces of degree $10$ in $\IP(1,1,2,5)$. However, it turns out that 
  this set coincides with $\gothD\gothC$, that is, no hypersurface of degree $10$ that passes through the $\frac 1 5(1,1,2)$ singularity of $\IP(1,1,2,5)$ has slc singularities. The argument, which was explained to us by Stephen Coughlan, runs as follows: 
  
  Let $\Delta$ be a disc with parameter $t$ and let $\IP = \IP(1,1,2,5)\times \Delta$. A general surface passing through the $\frac 1 5(1,1,2)$-singularity can be viewed as the central fibre $\mathcal{X}_0$ of a family of surfaces $\mathcal{X}\subset \IP$ with equation
  \[ tz^2 +f_5(x_0, x_1, y, t) z+f_{10}(x_0, x_1, y, t)=0,\]
  where $f_d$ is general of weighted degree $d$. We may assume that  for $t\neq0$ the fibres $\mathcal{X}_t$ are smooth.
  If $\mathcal{X}_0$ has slc singularities then  by \cite[Thm.~5.1]{ksb88} the total space $\mathcal{X}$ has canonical singularities. But  near $P_0=((0:0:0:1), 0)$ the $z$-coordinate is invertible and we can isolate $t$ in the local equation. In other words, $x_0, x_1, y$ are local orbifold coordinates near $P_0\in \mathcal{X}$ and $\mathcal{X}$ has a singularity of type $\frac 1 5(1,1,2)$ at $P_0$. By the Tai-Reid criterion (see \cite{Reid80} or \cite[Thm.~3.21]{KollarSMMP}),  this quotient singularity is not canonical, thus $\mathcal{X}_0$ cannot be slc.
  
  Indeed, Coughlan shows that the stable limit of the above family is a double cover of $\kq$ with Hurwitz divisor given by $f_5^2 =0$, i.e., a non-normal surface with normalisation two copies of $\kq$.
 \end{rem}

\subsubsection{ General surfaces in $\gothD\gothC\setminus \overline\gothM_{1,3}^{(Gor)}$}
As in Proposition \ref{prop: double-cover}, let $(\kq, \frac 12 \Delta)$ be a lc pair, where $\kq$ is the quadric cone and $\Delta$ is a quintic section. Let $p\colon X\to \kq$ be the associated double cover, which is a stable surface. It remains to treat the non-Gorenstein case, that is, the case where $\Delta$ contains the vertex $O$ of the cone.

Let $q\colon\kq'\isom \IF_2\to \kq$ be the blow up of the vertex, denote by $C$ the exceptional curve and write 
\[q^*(K_\kq +\frac 12 \Delta) = K_{\kq'} + \Delta'= K_{\kq'} + \frac 12 (\inverse q)_* \Delta + \Big( \frac{C\cdot(\inverse q)_* \Delta}{4} \Big) C, \]
where the coefficient of $C$ is computed by intersecting with $C$. Since $(\kq, \frac12\Delta)$ is lc if and only if $(Q', \Delta')$ is lc, we see that the strict transform of $\Delta$ intersects $C$ at most with multiplicity 4. Writing out the classes in the N\'eron-Severi group of $\kq'$ it is easy to check that the intersection is also even, which leaves us with three cases. 
  
We will now analyse $X$ via the commutative square
\[ \begin{tikzcd}
    X'\rar{q'}\dar{2:1}[swap]{p'} & X\dar{2:1}[swap]{p}\\
    \kq'\rar{q}& \kq
   \end{tikzcd}, 
\]
under the assumption that $\Delta$ is sufficiently general, in particular, $(\inverse q)_* \Delta $ is smooth and intersects $C$ transversally.

\textbf{ $C.(\inverse q)_* \Delta=0$:} This is the case where $\Delta$ does not contain the vertex. Note however, that since $O$ is in the branch locus of $p$, the double cover $p'$ is branched over $(\inverse q)_* \Delta +C$. The inverse image of $C$ is a $(-1)$-curve in $X'$ that is contracted to a smooth point of $X$, hence $X$ is a smooth surface of general type.

\textbf{ $C.(\inverse q)_* \Delta=2$:} Viewed as a degeneration of the previous case, the double cover $p'$ should be branched over $(\inverse q)_* \Delta +2C$. However, this would not be normal. The normalisation is the double cover branched over $(\inverse q)_* \Delta$ and thus the preimage of $C$ becomes a smooth rational curve of self-intersection $-4$. Thus $X$ has a $\frac 14(1,1)$ singularity above the vertex of the cone and is smooth otherwise  if $\Delta$ is sufficiently general.

If we let $F$ be a fibre of the ruling on $\kq'$ then $F.(\inverse q)_* \Delta=4$ and $2K_{\kq'}+(\inverse q)_* \Delta \sim 2F$. Hence $X'$ is a properly elliptic surface of Kodaira dimension 1 with $\chi(X')=\chi(X)=3$, because $X$ has rational singularities.

\textbf{ $C.(\inverse q)_* \Delta=4$:} Viewed as a degeneration of the previous cases, the double cover $p'$ should be branched over $(\inverse q)_* \Delta +3C$. Passing to the normalisation we get a  double cover branched over $(\inverse q)_* \Delta+C$.

 Since we chose $\Delta$ generic, $\Delta+C$ has 4 ordinary nodes along $C$; let $q''\colon \kq''\to \kq'$ be the blow up in these 4 points
 and consider the corresponding double cover.  
 Again passing to the normalisation, we obtain a double cover whose  branch divisor on $\kq''$ is the strict transform of $(\inverse q)_*\Delta+C$. 
On this double cover the configuration of curves we want to contract to $X$ has dual graph 
 \[\begin{tikzpicture}[auto,node distance=1cm,
  thick]
\node (1) {3};
\node (a1) [below left of=1] {2};
\node (a2) [below right of=1]{2};
\node (b1) [above left of=1]{2};
\node (b2) [above right of=1]{2};
  
  \path
    (1) edge  (a1)
    (1) edge  (a2)
        (1) edge  (b1)
            (1) edge  (b2);
\end{tikzpicture},
\]
where the numbers indicate  the negative self-intersection. The contraction of these curves gives a $\IZ/2$-quotient of an elliptic singularity on $X$
(see e.g. \cite[Ex.~3.28]{KollarSMMP}).

Arguing as in the previous case, the minimal resolution of $X$ is a properly elliptic surface with holomorphic Euler-characteristic equal to $3$.

\subsubsection{Surfaces which are not canonically embedded as hypersurfaces}
We now consider an example in $ \overline\gothM_{1,3} \setminus \gothD\gothC$.
\begin{example}
Consider, in the notation of Section \ref{section:canonical ring}, the ring $S_1[u]$ where $u$ has degree $2$.

Pick a family of surfaces depending on a parameter $t$
\[ X_t = \Proj S_1[u]/(f,g)\into \Proj S_1[u] \isom \IP(1,1,2,2,5)\]
where  $g = x_0x_1-tu$ and $f\in S_1[u]$ is general of degree $10$.

Then for $t\neq 0$ the surface $X_t$ is a Gorenstein stable surface as in  Theorem \ref{thm: canonical-ring}, \refenum{i},  since 
 we can use the equation  $g$ to eliminate the new variable $u$.

For $t = 0$ the bicanonical map realizes $X_0$ as a double cover of the union of two planes  $Y_0\subset \IP^3$, branched on a curve of degree $5$ and over the intersection line $r$ of the two planes. That is, $X_0$ consists of two singular $K3$ surfaces with five nodes,  which are double planes with branch curve a line plus a general quintic. The  surfaces are glued along the strict transform of the line and every canonical curve is a non-reduced curve supported on the non-normal locus of $X_0$. In particular,   $X_0$  is not canonically embedded in $\IP(1,1,2,5)$.

The surface $X_0$ is not Gorenstein, and its 
 Cartier index is equal to $2$.
 These surfaces give a $26$-dimensional locus inside the moduli space: the linear system of quintics in the two planes that match on the intersection line is of dimension $35$ and the automorphism group of $Y_0$ has dimension $9$. 
 \end{example}

\section{Bestiarium in $\overline\gothM_{1,2}^{(Gor)}$ and $\overline\gothM_{1,2}$ }\label{sec: pg=1}
In this section we consider the  moduli space of Gorenstein  stable surfaces with $K^2=1$ and $\chi=2$ (i.e. $p_g=1$). 
We refer to \cite{catanese79} and \cite{todorov80} for the analysis of the classical case.

By  Proposition  \ref{prop: pluri-chi=2} in this case the bicanonical map $X$ is  a degree 4 cover of the plane. The general theory of quadruple covers has been studied by Hahn-Miranda \cite{Hahn-MIranda1999} and by Casnati-Ekedahl \cite{Casnati-Ekedahl1996},
  but it is quite complicated and  a  description  as detailed as the one given  in section \ref{ssec:strata}  is not feasible in this case.   Thus here we restrict to a much coarser analysis, just giving a cornucopia of examples, mostly in the special case when the bicanonical map is Galois. 
  Canonical surfaces with this property are also called Kunev surfaces after \cite{kynev77} and they have been studied in connection with the failure of local Torelli for surfaces. The bicanonical map of  a Kunev surface factors through a double cover of a (singular) $K3$-surface of degree two.  From this point of view degenerations of such surfaces have been studied by Usui, again with applications to Torelli-type questions in mind \cite{usui87, usui00}.

  An overview over the examples we construct can be found in Table \ref{tab: normal strata pg=1}.

  \begin{table}\caption{Some normal and non-normal surfaces in $\overline\gothM_{1,2}^{(Gor)}$}\label{tab: normal strata pg=1}
 \begin{tabular}{ccclc}
 \toprule
type of cover  & example &normal   & minimal resolution $\tilde X$ & $\kappa(\tilde X)$\\
\midrule
\multirow{7}*{bi-double} &
$Z_1$ & yes & minimal elliptic surface 
& $1$\\
&$Z_{1,1}^A$&yes & blow up of an abelian surface & $0$\\
&$Z_{1,1}^B$& yes &blow up of a bielliptic  surface & $0$\\
&$Z_{4}$& yes &rational surface & $-\infty$\\
&$Z^{(dP)}$& no & del Pezzo surface of degree 1 & $-\infty$\\
&$Z^{(E_-)}$& no &  ruled with $\chi(\tilde X) = 0$ & $-\infty$\\
&$Z^{(P)}$& no &  $\IP^2$  & $-\infty$\\
\midrule
\multirow{2}*{iterated double}
& $Z_2^E$ & yes &blow up of an Enriques surface &$0$\\
& $Z_2^R$ &yes & rational surface &$-\infty$\\
\bottomrule
\end{tabular}

\end{table}

\subsection{Gorenstein bi-double covers of the plane}\label{sec: bidouble}
  In most of our  examples the bicanonical map is a  $\IZ_2^2$-cover (``bi-double cover''): it is not hard to show that  this is the case   when the terms $a_1$ and $a_2$ in the equations of Theorem \ref{thm: canonical-ring}, \refenum{ii} vanish (see \cite[\S 1, Prop. 10]{catanese79} for the smooth case).   

Non-normal abelian covers are studied in  \cite{alexeev-pardini12};
 we recall below the facts  that we need in the special case of a bi-double cover of $\phi\colon X\to \pp^2$. 
 If $X$ is demi-normal, then by \cite[Cor.~1.10]{alexeev-pardini12} $\phi$ is uniquely determined (up to isomorphism of covers) by effective divisors $D_i$ of $\pp^2$ of degree $d_i$, $i=0,1,2$,  such that:
 \begin{itemize}
 \item  $d_i\equiv d_j \mod 2$ for every $i,j$.
 \item the so-called Hurwitz divisor $\Delta:=\frac 12(D_0+D_1+D_2)$ has no component of multiplicity $>1$. 
 \end{itemize}
The divisors  $D_0$, $D_1$ and $D_2$ are called  the branch data of $\phi$; setting  $a_i=\frac{d_j+d_k}{2}$, where $i,j,k$ is a permutation of $0,1,2$, one has: 
 $$\phi_*\OO_X=\OO_{\pp^2}\oplus \OO_{\pp^2}(-a_0)\oplus \OO_{\pp^2}(-a_1)\oplus\OO_{\pp^2}(-a_2).$$
 
   \begin{prop}\label{prop: bi-double} In the  above set-up:
\begin{enumerate}
\item $2K_X=\phi^*\OO_{\pp^2}(d_0+d_1+d_2-6)$
\item $X$ is slc if and only if $(\pp^2, \Delta)$ is an lc pair 
\item $X$ is non-normal above an irreducible curve  $\Gamma$ of $\pp^2$ if and only if $\Gamma$ appears in $\Delta$ with coefficient  $1$
\item $X$ is Gorenstein if and only if $D_0\cap D_1\cap D_2=\emptyset$.
\end{enumerate}
\end{prop} 
\begin{proof}
\refenum{i}  and \refenum{ii} follow by \cite[Prop.~2.5]{alexeev-pardini12} and \refenum{iii} follows by \cite[Thm.~1.9, (2)]{alexeev-pardini12}.

To prove \refenum{iv} observe first of all that if a point $P\in \pp^2$  lies on at most two of the $D_i$ then locally above $P$ the map $\phi$ can be seen as the composition of two flat double covers, hence $X$ is Gorenstein above $P$.
 If $P$ lies on all the  $D_i$ and $D_0+D_1+D_2$ has an ordinary triple point at $P$, then $X$ has a singular point of type $\frac 14(1,1)$ at $P$ (see\cite[Table~1)]{alexeev-pardini12},  case 4.3). 
Since being Gorenstein is an open condition \cite[Cor.~3.3.15]{Bruns-Herzog}
it follows that $X$ is not Gorenstein at $P$ for any choice of divisors $D_i$ through $P$.
\end{proof}
By Proposition \ref{prop: bi-double}, to construct a stable surface with $K^2_X=1$ and $\chi(X)=2$ as a bi-double  cover of $\pp^2$ one has to choose as branch data a line $D_0$ and cubics $D_1$, $D_2$ such that $(\pp^2, \frac 12(D_0+D_1+D_2))$ is an lc pair: 
indeed, since $K^2=1$ we must have $d_0+d_1+d_2 = 7$   and thus all the $d_i$ have to be odd by the parity condition. Computing $2=\chi(X)=\chi(\phi_*\ko_X) = 4+\frac 1 2 \sum a_i(a_i-3)$ we see that the only possibility is $(d_0, d_1, d_2)=(1,3,3)$, that is, $(a_0, a_1, a_2) = (3,2,2)$. 
In particular, the projection formula implies
$|2K_X|=\phi^*|\OO_{\pp^2}(1)|$.

\begin{rem}\label{rem: normalise} Let $\phi\colon X\to\pp^2$ be an  slc  bi-double cover with branch data $D_i$, $i=0,1,2$. 
By Proposition \ref{prop: bi-double} the surface  $X$  is not normal above an irreducible curve $\Gamma$ of $\pp^2$ if and only if one of the following happens: 
\begin{itemize}
\item[(a)]  $\Gamma$ appears with multiplicity 2 in exactly one of the $D_i$
\item[(b)] $\Gamma$ appears with multiplicity 1 in exactly two  of the $D_i$.
\end{itemize}
Note that if case (b) occurs, then $X$ is not Gorenstein by Proposition \ref{prop: bi-double}, \refenum{iv}. The normalisation algorithm (see \cite[\S 3]{pardini91}) is very simple in this situation: in case (a) one subtracts $2\Gamma$ from the only divisor containing it, while in case (b) one subtracts $\Gamma$ from the two divisors containing it and adds it to the remaining one. In both cases the divisors thus obtained are the branch data of a cover $\phi'\colon X'\to\pp^2$ such that  $X'$ is normal above $\Gamma$ and there is a birational morphism $X'\to X$ commuting with the covering maps to $\pp^2$.

More generally, if $\phi$ is a standard bi-double  cover (see \cite[\S~1]{alexeev-pardini12} for the definition) with branch data $D_i$ then one   first reduces the $D_i$ modulo 2 and then  removes  all the irreducible components common to all the $D_i$. This way one obtains a cover such that $\Delta$ has no component of multiplicity $>1$, whose normalization can be computed as in the slc case. 
\end{rem}

We now  give  our examples   by describing   the branch data $D_i$. 
We take coordinates $(y_0,y_1,y_2)$ in $\pp^2$ so that 
the first branch divisor is the line $D_0=\{y_0=0\}$ and we only specify the cubics  $D_1$ and $D_2$.
The  possible singularities  of slc $(\IZ/2)^r$-covers such that the support of the Hurwitz divisor $\Delta$  has ordinary singularities have been classified in \cite[Table~1--4]{alexeev-pardini12}. In our restricted situation only two different normal singularities 
can occur, because $r=2$ and  $D_1$  and $D_2$ can have at most  three local branches through a point $P$:
\begin{itemize}
 \item $D=2\Delta$ has an ordinary quadruple point at $P$, such that three of the local components are in the same $D_i$. The resulting singularity is an elliptic singularity of degree $1$ (see case 4.5 in loc.cit.).
 \item Both $D_1$ and $D_2$ have an ordinary double point at $P$ such that $D$ has an ordinary quadruple point at $P$. The resulting singularity is an elliptic singularity of degree $4$ (see case 4.6 in loc.cit.).  
\end{itemize}

The low degree of the branch divisors leaves very few combinatorial possibilities. We will now describe all possible normal examples where $D$ has only ordinary singularities; the surface $Z_{d_1, \dots, d_r}$ will have $r$ elliptic singularities with the given degrees.
Often the identification of the minimal resolution $\pi\colon \tilde X \to X$ is immediate by the restrictions found in   \cite[Thm.\ 4.1]{FPR15a}.

\begin{description}
 \item[Example $Z_1$] Let $D_1$ be a union of three general lines through $P\in D_0$ and $D_2$ a general cubic. Then $X$ has a unique elliptic singularity of degree $1$.
 Blowing up at $P$ and then changing the branch divisor to get a normal bi-double cover $\tilde X \to \IP^2$ as in Remark \ref{rem: normalise}, one computes that $|2K_{\tilde X}|$ is an elliptic pencil, induced by the pencil of lines passing through $P$. Thus by \cite[Thm.\ 4.1]{FPR15a} $\tilde X$ is a minimal properly elliptic surface with $\chi(\tilde X) = 1$. 
 
 \item[Example $Z_1'$] If in the previous example the point $P$ is a general point on $D_1$ instead of  on $D_0$, we get, by the same computation,  a different family of surfaces with minimal resolution a minimal properly elliptic surface with $\chi(\tilde X) = 1$. 
 
 \item[Example $Z^A_{1,1}$]  Let $P, Q$ be two different points on the line $D_0$ and let $D_1$ be a union of three general lines through $P$ and $D_2$  be a union of three general lines through a $Q$. Then $X$ has  two elliptic singularities of degree $1$. Blowing up at $P$ and $Q$ and then changing the branch divisor to get a normal bi-double cover as in Remark \ref{rem: normalise}, one computes  that  $2K_{\tilde X}$ is linearly equivalent to the strict transform of  $ \phi_2^*D_0$ on $\tilde X$, which is twice a $(-1)$-curve. Thus $\tilde X$ is a blow-up of a surface with trivial bicanonical bundle and $\chi(\tilde X)=0$.  A formula to compute the sections of the canonical bundle of a singular bi-double cover has been given in \cite[Sect.~3]{catanese99}. Using it it is straightforward to check that $K_{\tilde X}$ is effective as well and hence the minimal model of $\tilde X$ is an abelian surface.

 \item[Example $Z^B_{1,1}$] Let $D_1$ be a union of three general lines through $P\in D_0$ and $D_2$  be a union of three general lines passing through a general point  $Q \in D_1$. Then $X$ has  two elliptic singularities  of degree $1$. Blowing up at $P$ and $Q$  and then changing the branch divisor to get a normal bi-double cover 
 $\tilde X \to \IP^2$ as in Remark \ref{rem: normalise}, one computes  that  $2K_{\tilde X}$ is linearly equivalent to the strict transform of  $ \phi_2^*D_0$ on $\tilde X$, which is twice a $(-1)$-curve. Thus $\tilde X$ is a blow-up of a surface with trivial bicanonical bundle and $\chi(\tilde X)=0$.

Using again \cite[Sect.~3]{catanese99} one checks that $K_X$ is not effective and hence the minimal model of $\tilde X$ is a bielliptic surface.
 
 \item[Example $Z_{4}$] Assume that both $D_1$ and $D_2$ have an ordinary node at a point $P$ not in $D_0$. Then $X$ has an elliptic singularity of degree $4$ and thus by  \cite[Thm.\ 4.1]{FPR15a} its minimal resolution is a rational surface. Indeed, the pull back of the pencil of lines through $P$ gives a free pencil of rational curves on the resolution.
 \end{description}
 
We now turn to the non-normal case.  By Proposition \ref{prop: bi-double}  $X$ is  non-normal and Gorenstein  if and only if there exists an irreducible curve  $\Gamma$ that appears with multiplicity 2 in one of the $D_i$  and is not contained in the remaining two. In particular, $\Gamma$ must be a line. 
This leaves very few possibilities that we describe below. We denote by $\bar X$ the normalisation of $X$; the possibilities for $\bar X$  are listed in  Section \ref{ssec: definitions}. 
\begin{description}
 \item[Example $Z^{(dP)}$] If $D_1=2L_1+L_2$ is the union of a double line and a general line and $D_2$ is general then $-2K_{\bar X} = \phi_2^*(-2K_{\IP^2}-D_0-L_2-D_2) \sim \phi_2^*L_2$ is ample of square $4$. Thus $\bar X$ is a del Pezzo surface of degree $1$ and  $X$ is of type $(dP)$.
 \item[Example $Z^{(dP)}_1=Z^{(E_-)}$]   If $D_1=2L_1+L_2$ is the union of a double line and a general line and $D_2$ is the union of three lines meeting $D_0$ at a general point, then $X$ acquires an additional elliptic singularity of degree $1$. Thus the normalisation is a singular del Pezzo of degree $1$ with minimal resolution a ruled surface over an elliptic curve: $X$ is of type $(E_-)$.
 \item[Example $Z^{(P)}$] If both $D_1$ and $D_2$ contain a double line then the normalisation $\bar X$  of $X$ is the projective plane. Indeed the induced cover $\bar X \to \IP^2$ is given by squaring the coordinates. 
 \end{description}
\begin{rem}\label{rem: E+}
Since the normalisation $\bar X$ of a non-normal bi-double cover of $\IP^2$ is again such a bi-double cover, the canonical divisor of $\bar X$ is the  pullback of some $\ko_{\IP^2}(d)$ and thus either ample, anti-ample, or trivial. Thus no bi-double cover can have normalisation of type $(E+)$.

A construction of a different flavor which relies on a glueing result of Koll\'ar has already been given in \cite[Sect. 3.3]{FPR15a}: let $E$ be an elliptic curve and $\bar D\subset S^2 E$ a general 3-section of the Albanese map. Then $\bar D$ is a smooth curve of genus 2. To get a non-normal stable surface, we glue $\bar D$ to itself via the hyperelliptic involution $\tau$. The result is a surface $X$ which is non-normal along $D=\bar D/\tau$. At a general point of $D$ the surface $X$ has a double normal-crossing singularity while at the branch points of $\bar D\to D$ we find pinch points with local equation $z^2-yx^2=0$.

It would be interesting to compute the canonical ring directly from this description, thus realizing $X$ as a complete intersection in $\IP(1,2,2,3,3)$. 
\end{rem}

\subsection{Iterated double covers}
We start by noting that a  bi-double cover with branch data $D_0$, $D_1$, $D_2$ can be seen as an iterated double cover as follows.  First one takes the double cover $f\colon Y\to \pp^2$ branched on $D_0+D_1$: if $D_0$ and $D_1$ intersect transversally, then  $Y$ is a singular Del Pezzo surface of degree 2 that has  ordinary double points over the three  intersection points of $D_0$ and $D_1$. The cover $g\colon X\to Y$ is obtained by taking the double cover of $Y$  branched on the singular points and on the divisor $B:=f^*D_2\in |-3K_Y|$. 

More generally,  the same construction can be performed taking as $B$ any element of the system $|-3K_Y|$ not passing through the singular points of $Y$; in this way one obtains a Gorenstein cover $\phi_2=g\circ f\colon X\to\pp^2$ with the same numerical invariants, which in general is not Galois.

We give two  variants of this construction by specifying the plane cubic  $D_1$ and the branch divisor $B$. It is possible to impose further or different singularities either on $D_0+D_1$ or on $B$ to get other examples but we will not pursue this here.

\begin{description}
  \item[Examples $Z_2^R$ and $Z_2^E$] 
 By taking $D_1$ general and choosing  $B$ 
 with a quadruple point a smooth point $Q$ of $Y$  such that  the infinitely near points are at most double 
 we obtain an example with an elliptic Gorenstein singularity of degree 2.
 The minimal desingularization $\wt X$ of $X$ has $\chi(X)=1$, hence by \cite[Thm.\ 4.1]{FPR15a} it is either rational or birational to an Enriques surface.  By the standard formulae for double covers, the  bicanonical curves of $\wt X$  correspond to the curves in $|-K_Y|$ with a double point at $Q$. So we have $P_2(\wt X)>0$  if $Q$ lies on the ramification divisor of $Y\to \pp^2$  and $P_2(\wt X)=0$ otherwise. So $\wt X$ is Enriques in the former case ($Z_2^E$) and it is rational ($Z_2^R$) in the other one.

 That such a divisor $B$ exists can be seen as follows: for $Z_2^E$ pick a general point $P$ on $D_1$, let $l_0$ be an equation of the tangent  to $D_1$ at $P$ and let $l_1, l_2$ be the equations of general lines through $P$.  Then the curve $B$ of equation $l_0l_1l_2=0$  has a quadruple point at the preimage $Q$ of $P$, with an infinitely near double point. Taking the  double cover branched on $B$ one obtains   an elliptic singularity of degree 2, whose exceptional cycle consists of a $-4$-curve and a $-2$-curve meeting transversally at two points.

  For $Z_2^R$ one can for instance take
 $Y = \{y^2-x_0(x_0^3+x_1^3+x_2^3+2x_0x_2^2)=0\}$
 and $B\subset Y$ given by  $ \{ x_1(y+x_0^2+x_2^2)=0\}$, which has the required quadruple point at $Q=(1:0:0:-1)$ and no other singularities.
 \end{description}

\subsection{Some non-Gorenstein examples}\label{sec: non-Gorenstein chi=2}
We conclude this section by giving two examples of non-Gorenstein Galois-covers of the plane.

The first one is a bidouble cover which occurs as a degeneration of the construction in Section \ref{sec: bidouble} but the second one is a (non-simple) cyclic cover with Galois-group $\IZ/4$. This cannot occur in the classical case and we do not know if it is contained in the closure of $\gothM_{1,2}$ in $\overline\gothM_{1,2}$.

\begin{example}
By Proposition \ref{prop: bi-double}, a non Gorenstein degeneration of a bi-double cover of the plane of the type analysed  in Section \ref{sec: bidouble} can be obtained by letting the divisors $D_i$ all go through a point $P$. If the $D_i$ are taken to be general otherwise, then $X$ has a singularity of type $\frac 14(1,1)$ over $P$ and is smooth elsewhere. The bicanonical system of the minimal resolution $\wt X$ is a free linear pencil of elliptic curves (the strict transform   of the  pencil of lines through $P$), hence $\wt X$ is properly elliptic. More degenerate configurations can be analysed as above.

We just briefly describe here the additional possibilities for the  non-normal case assuming that  the components of $\Delta$ are general (we  keep the previously  introduced notation):
\begin{itemize}
\item  $D_1$ and $D_2$ have a line in common. In this case  the normalisation $\bar X$ is an Enriques surface with two $A_1$ points.  
\item  $D_1$ and $D_2$ have a conic in common. In this case $\bar X$ is a singular del Pezzo surface of degree 1with four $A_1$ points.
\item $D_0$ is a component of $D_1$. In this case $\bar X$ is a singular $K3$  surface  with six  $A_1$ points.
\end{itemize} 
Notice that  of the  normalisations that we obtain in the first and the last case  cannot  occur in the Gorenstein case (cf. Section \ref{ssec: definitions}).
\end{example}

Finally we give an example such that the bicanonical map is a $\IZ_4$-cover:
\begin{example}

Let $D_1$ and  $D_2$ be lines and $D_3$ be a reduced cubic of $\pp^2$. If we take $L=\OO_{\pp^2}(3)$, the equivalence relation
$4L\equiv D_1+2D_2+3D_3$ is satisfied and therefore by \cite[Prop.~2.1]{pardini91} there exists a $\IZ_4$-cover $\phi\colon X\to \pp^2$ such that: 
\begin{itemize}
\item the preimages $R_1$ and $R_3$ of  $D_1$ and $D_3$  are fixed pointwise by  $\IZ_4$;  the group acts  on the normal space to $R_1$  at  a general point via a character $\chi$ of order 4 and on the normal space to the preimage of $R_3$  at  a general point via the opposite character $\chi\inv$.
\item the preimage $R_2$ of $D_2$ is fixed pointwise by the order 2 subgroup of $\IZ_4$ but not by all the group. 
\item  $\phi_*\OO_X=\OO_{\pp^2}\oplus \OO_{\pp^2}(-2)\oplus \OO_{\pp^2}(-2)\oplus\OO_{\pp^2}(-3)$.
\end{itemize}
In this case the Hurwitz divisor is $\Delta=\frac 34 (D_1+D_3)+ \frac 12 D_2$, hence $2K_X=\phi^*(2K_{\pp^2}+2\Delta)=\phi^*(\OO_{\pp^2}(1))$ and the projection formula  gives that $|2K_X|=\phi^*|\OO_{\pp^1}(1)|$, hence $X$ is $2$-Gorenstein. 
For a general choice of the $D_i$ the singularities of $X$ are three points of type $\frac 14(1,1)$, occurring over the intersection points of $D_1$ and $D_3$, and four points of type $A_1$ occurring over the intersection points of $D_2$ with $D_1+D_3$. In particular,  $X$ is not Gorenstein. 

An interesting feature of this example is that by \cite[Thm.~1.9, (2)]{alexeev-pardini12} if $X$ is demi-normal then it is normal, namely one cannot obtain non-normal examples by degenerating this construction.

These examples give an $8$-dimensional subset of $\overline\gothM_{1,2}$ but we do not know whether this set lies in the closure of $\overline\gothM_{1,2}^{(Gor)}$.
\end{example}


 \end{document}